\documentclass{amsart}
\usepackage{latexsym, amssymb, mathrsfs, amsthm, mathtools, amsfonts, amsmath, tikz-cd, xcolor, enumitem}
\usepackage[all,cmtip]{xy}
 \usepackage{appendix}
 \usepackage{xy}
\usepackage{hyperref}
 \usepackage{cleveref}
 
\def\ddefloop#1{\ifx\ddefloop#1\else\ddef{#1}\expandafter\ddefloop\fi}

\def\ddef#1{\expandafter\def\csname bb#1\endcsname{\ensuremath{\mathbb{#1}}}}
\ddefloop ABCDEFGHIJKLMNOPQRSTUVWXYZ\ddefloop

\ddefloop ABCDEFGHIJKLMNOPQRSTUVWXYZ\ddefloop

\def\ddef#1{\expandafter\def\csname ff#1\endcsname{\ensuremath{\mathfrak{#1}}}}
\ddefloop ABCDEFGHIJKLMNOPQRSTUVWXYZ\ddefloop
\def\ddef#1{\expandafter\def\csname cc#1\endcsname{\ensuremath{\mathcal{#1}}}}
\ddefloop ABCDEFGHIJKLMNOPQRSTUVWXYZ\ddefloop

\newcommand{\Spec}{\mathrm{Spec}}
\newcommand{\Proj}{\mathrm{Proj}}
\newcommand{\Sym}{\mathrm{Sym}}
\newcommand{\id}{\mathrm{id}}
\newcommand{\Hom}{\mathrm{Hom}}
\newcommand{\sHom}{\mathrm{\underline{Hom}}}
\newcommand{\Sec}{\mathrm{\underline{Sec}}}
\newcommand{\dHom}{\bbR\mathrm{\underline{Hom}}}
\newcommand{\pre}{\mathrm{pre}}
\newcommand{\ev}{\mathrm{ev}}

\newcommand{\vir}{\mathrm{vir}}
\newcommand{\Kos}{\mathrm{Kos}}
\newcommand{\DAG}{\mathrm{DAG}}
\newcommand{\Coh}{\mathrm{Coh}}
\newcommand{\Perf}{\mathrm{Perf}}

\newcommand{\wC}{\widehat{\ffC}}

\newcommand{\uC}{\underline{\ffC}}
\newcommand{\wL}{\widehat{\ffL}}

\newcommand{\Pic}{\mathfrak{Pic}^{\mathrm{s}}}
\newcommand{\Picc}{\mathfrak{Pic}}
\newcommand{\wPic}{\widehat{\Pic}}
\newcommand{\wPicc}{\widehat{\Picc}}

\newcommand{\RQm}{\bbR\overline{\ccQ}_{g,n} (\bbP^r,d)}
\newcommand{\Qm}{\overline{\ccQ}_{g,n} (\bbP^r,d)}
\newcommand{\RGw}{\bbR\overline{\ccM}_{g,n}(\bbP^r,d)}
\newcommand{\Gw}{\overline{\ccM}_{g,n}(\bbP^r,d)}

\newcommand{\SecM}{\bbR\Sec_{\Pic}(\ffL^{\oplus r+1}/\ffC)}
\newcommand{\SecQ}{\bbR\Sec_{\wPic}(\wL^{\oplus r+1}/\wC)}

\newcommand{\upi}{\underline{\pi}}

\newcommand{\uc}{\underline{c}}
\newcommand{\uk}{\underline{k}}

\newcommand{\oc}{\overline{c}}

\newcommand{\UU}{\;U \mkern-21mu U\;}

\newcommand{\rder}{\mathbf{R}}
\newcommand{\lder}{\mathbf{L}}

\newcommand{\wh}[1]{\widehat{#1}}
\newcommand{\un}[1]{\underline{#1}}

\newtheorem{theorem}[subsubsection]{Theorem}
\newtheorem{corollary}[subsubsection]{Corollary}
\newtheorem{lemma}[subsubsection]{Lemma}
\newtheorem{proposition}[subsubsection]{Proposition}
\theoremstyle{definition}
\newtheorem{definition}[subsubsection]{Definition}

\theoremstyle{remark}
\newtheorem{remark}[subsubsection]{Remark}
\newtheorem{example}[subsubsection]{Example}
\newtheorem{construction}[subsubsection]{Construction}
\newtheorem{claim}[subsubsection]{Claim}
\title{Derived moduli of sections and push-forwards}
\author{David Kern}
\author{\'Etienne Mann} 
\author{Cristina Manolache}
\author{Renata Picciotto}

\address{David Kern, Univ Montpellier, CNRS, IMAG, F-34000 Montpellier, France}
\email{david.kern@umontpellier.fr}
\address{\'Etienne Mann, Univ Angers, CNRS, LAREMA, SFR MATHSTIC, F-49000 Angers, France}
\email{etienne.mann@univ-angers.fr}
\address{Cristina Manolache, Univ Sheffield, J23 Hicks building, Sheffield, UK}
\email{c.manolache@sheffield.ac.uk}
\address{Renata Picciotto, Univ Angers, CNRS, LAREMA, SFR MATHSTIC, F-49000 Angers, France}
\email{renata.picciotto@univ-angers.fr}

\begin{document}
\maketitle

\begin{abstract}
    We introduce a derived enhancement $\bbR\Sec_{\ffM}(\ffZ/\ffC)$ of the moduli space of sections $\Sec_{\ffM}(\ffZ/\ffC)$ defined by Chang--Li, and we compute its tangent complex. Special cases of this moduli space include stable maps and stable quasi-maps. As an application, we prove that $G$-theoretic stable map and quasi-map invariants of projective spaces are equal.
\end{abstract}
\tableofcontents
\section{Introduction}

{\bf Statements of the main results.}
This work consists of two parts: in the first part we construct a derived moduli of sections, which encodes both the classical and the virtual geometry of many well-studied moduli spaces, such as moduli of stable maps and quasi-maps (see Theorem \ref{thm:intro,derived,section}); in the second part we suggest a new approach to virtual push-forward formulae via derived geometry, in form of the following theorem.
\begin{theorem}
Let $g,n,r$ and $d$ be positive integers and let $\bbR\overline{\ccM}_{g,n} (\bbP^r,d)$ and $\bbR\overline{\ccQ}_{g,n} (\bbP^r,d)$ denote the derived moduli space of genus $g$, degree $d$ stable maps, respectively quasi-maps to a projective space $\bbP^r$.
\begin{enumerate}
    \item We have a derived morphism 
\[
\oc: \RGw \to \RQm
\]
and an isomorphism
\[
\oc_* \ccO_{\bbR\overline{\ccM}_{g,n} (\bbP^r,d)} = \ccO_{\bbR\overline{\ccQ}_{g,n} (\bbP^r,d)} \mbox{ in } \ccD^b(\Coh(\bbR\overline{\ccQ} (\bbP^r,d)).
\]
\item Consequently, $G$-theoretic stable map and quasi-map invariants are the same (See Corollary \ref{coro:K,GW,QM}).
\end{enumerate}
\end{theorem}
The second part of the theorem generalises the already known cohomological result (see \cite{Fontanine-QMap2010}, \cite[Theorem 3]{Marian-Oprea-Pand-moduli-stable-2011} and \cite[Proposition 5.19]{Manolache-virtual-push-2012}): 
\[
t_0(\oc)_* [\overline{\ccM}_{g,n} (\bbP^r,d)]^\vir = [\overline{\ccQ}_{g,n} (\bbP^r,d)]^\vir \mbox{ in } A_*(\overline{\ccQ}_{g,n} (\bbP^r,d))\text{.}
\]
This shows that our statement is a categorification of  the equality above.

The derived statement is obtained by studying the contraction map $\oc$ locally, after constructing compatible derived atlases on the two spaces. The advantage over the classical situation is that local information can now be used to obtain global statements: rather than having external information based on choices of perfect obstruction theories, this data is now encoded in the geometry of the derived moduli spaces.

In the following we introduce the moduli of sections (see Chang--Li \cite[\S 2]{CL12}), which is the central object of study in this paper. Consider an Artin stack $\ffM$ with a flat, nodal, projective curve $\ffC$ and a morphism of $\ffM$-Artin stacks $\pi:\ffZ\to \ffC$. For any test scheme $S\to \ffM$ the moduli of sections is defined as 
\[
\Sec_{\ffM} (\ffZ/\ffC)(S\to\ffM)=\left\{ f:C_S\to Z_S | \pi_S\circ f=id_{C_S} \right\},
\]
for $\pi_S: Z_S:=\ffZ\times_\ffM S \to C_S:=\ffC\times_\ffM S$. 
In order to construct  a quasi-smooth derived enhancement of this space, we additionally need to require that the Artin stack $\ffZ$ is smooth relative to $\ffC$.

If $\ffM=\ffM^{\pre}_{g,n}$ is the moduli space of genus $g$, $n$-pointed prestable curves and $\ffC$ is its universal curve, we can take $\pi: \ffZ \to \ffC$ to be a trivial fibration $\ffZ:=\ffC\times_\ffM X$ where $X$ is a smooth projective variety or DM stack. Then $\Sec_{\ffM}(\ffC\times X/ \ffC)$ is the usual moduli space of stable maps to $X$ (See Example \ref{ex:stablemaps}). For nontrivial fibrations, this construction recovers moduli of quasi-maps and twisted theories such as stable maps with fields (see Example in \S \ref{ss:cone}).

In \S \ref{Derived-sections}, we put a natural derived structure, denoted by $\bbR\Sec_\ffM(\ffZ/\ffC)$, on the moduli of sections and we compute its relative tangent complex. This is compatible with the perfect obstruction theory defined by \cite{CL12} and vastly generalized by ~\cite{W21} (see \ref{ss:cone}). 
For a precise statement, consider the following universal family over $\bbR\Sec_{\ffM}(\ffZ/\ffC)$:
\begin{equation}\label{diag:ev,sec,intro}
    \begin{tikzcd}
    & \ffZ\ar[d,"p"]\\
     \ffC_{\bbR\Sec_{\ffM}(\ffZ/\ffC)}\arrow[ur,"\ev"]\ar[d,"\pi_{1}"']\ar[r]& \ffC\ar[d,"\pi"]\\
     \bbR\Sec_{\ffM}(\ffZ/\ffC)\ar[r]& \ffM.
    \end{tikzcd}
\end{equation}

With this, we have:
\begin{theorem}\label{thm:intro,derived,section}
For $\ffZ$ a locally almost finitely presented 1-Artin stack with quasi-affine diagonal relative to $\ffM$,
\begin{enumerate}
    \item Corollary \ref{cor:representability_sec}. The functor of sections of $\ffZ\to\ffC$ over $\ffM$ is representable by a locally almost finitely presented 1-Artin stack with quasi-affine diagonal $\bbR\Sec_{\ffM}(\ffZ/\ffC)\to\ffM$.
    \item Corollary \ref{coro:derived,Rsec}. The truncation of $\bbR\Sec_{\ffM}(\ffZ/\ffC)$ is the moduli of sections of introduced by Chang-Li in \cite{CL12} (see also \cite{W21}).
    \item Proposition \ref{prop:moduli_section_vector_bundle}. If $\ffZ$ is a vector bundle on $\ffC$ with sheaf of sections $\ccE$,  the moduli of sections is itself a derived vector bundle. More precisely, we have 
    \[
    \bbR\Sec_{\ffM}(\ffZ/\ffC)=\bbR \Spec_\ffM  \Sym (\rder\pi_* \ccE)^\vee\text{.}
    \] 
    \item Theorem \ref{thm:tanS}. Using the notation of Diagram \eqref{diag:ev,sec,intro}, the relative tangent space of $\bbR\Sec_{\ffM}(\ffZ/\ffC)$ is 
\[
\bbT_{\bbR\Sec_{\ffM}(\ffZ/\ffC)/\ffM}=\rder\pi_{1 *}\lder\ev^*\bbT_{\ffZ/\ffC}\text{.}
\]
\end{enumerate}
\end{theorem}

The moduli space of curves on projective spaces (or more generally on toric DM stacks) admits various compactifications, which are all substacks of a common moduli of sections. As a map to the projective space $\bbP^r$ is a line bundle with sections, the (underived) moduli of quasi-maps to $\bbP^r$, denoted by $\overline{\ccQ}_{g,n}(\bbP^r,d)$, and the moduli of stable maps, denoted by $\Gw$,  are both open substacks of a moduli of sections over $\Picc=\Picc_{g,n,d}$ -- the moduli space of line bundles over pre-stable curves. We thus obtain derived structures on the moduli space of stable maps and quasi-maps, denoted by $\RGw$ and $\bbR\ccQ_{g,n}(\bbP^r,d)$ respectively. In \cite{STV} the authors define another derived structure on $\Gw$, which is induced by a Hom-space over $\ffM^{\pre}_{g,n}$. In \S \ref{sec:compare_derived_stable_maps}, we prove that the derived structure in \cite{STV} and the derived structure described above are the same (see Theorem \ref{thm:der_structures_on_Mgn}).
\\

{\textbf {Outline of the paper.}}

In \S \ref{sec:notation} we introduce notation.

In \S \ref{Derived-sections} we give a natural derived structure on the moduli of sections (see Corollary \ref{coro:derived,Rsec}) and compute its tangent spaces (see Theorem \ref{thm:tanS}). We also investigate the case when $\ffZ$ is a bundle (see Proposition \ref{prop:moduli_section_vector_bundle}).

In \S \ref{sec:compare_derived_stable_maps} we study in detail the cases of the moduli of stable maps and quasi-maps viewed inside the derived stacks of sections. We prove that the two derived structures -- the one coming from the moduli of sections and the one from maps -- are the same (see Theorem \ref{thm:der_structures_on_Mgn}).

In \S \ref{sec:GWtoQM} (see Theorem \ref{prop:div,rational,tails}) we construct the derived morphisms 
\begin{equation}\label{derived c}
\oc :\bbR \Sec_{\Picc}(\ffL^{\oplus r+1}/\ffC) \to \bbR \Sec_{\wPicc} (\widehat{\ffL}^{\oplus r+1}/\widehat{\ffC})\text{,}
\end{equation}
where $\Picc$ denotes the stack parametrising pre-stable curves together with a line bundle and $\wPicc$ denotes the stack parametrising pre-stable curves without rational tails\footnote{Rational tails are tree of $\bbP^1$ that does not have marked point. See Definition \ref{defi:divisor,rational,tails} for details.} together with a line bundle. Over $\Picc$ we have a universal curve and a universal line bundle:
\[
\ffL \to \ffC \to \Picc \text{.}
\]
Similarly, over $\wPicc$ we have a universal family
\[
\widehat{\ffL} \to \widehat{\ffC} \to \wPicc \text{.}
\]
The truncation of morphism (\ref{derived c}) recovers the map 
\[
\overline{c}:\Gw\to \Qm
\]
defined in \cite{Fontanine-QMap2010}, \cite[Theorem 3]{Marian-Oprea-Pand-moduli-stable-2011} and \cite[Proposition 5.19]{Manolache-virtual-push-2012}, which contracts rational tails.

In \S \ref{subsec:local,embeddings} we prove that the pushforward by $\oc$ of the derived structure sheaf of the moduli space of stable maps is the structure sheaf of the space of quasi-maps (see Theorem \ref{thm:pushforward}). The main idea is to find compatible derived atlases on the two spaces. On $\RGw$ each chart consists of a triple $(W,F,\theta)$, where $W$ is a smooth stack over $\Picc$, $F$ is a vector bundle over $W$ and $\theta$ is a section of $F$ such that locally $\RGw\simeq Z^h(\theta)$. Here $Z^h(\theta)$ denotes the derived vanishing locus of $\theta$. We construct a similar atlas for $\RQm$. 
\\

{\bf Historical note.}
Moduli spaces appearing in Gromov--Witten theory and, more broadly, in enumerative geometry, are usually singular and they may have irreducible components of different dimensions. To extract information about enumerative problems, such as various types of invariants, one needs to integrate over these moduli spaces. As such spaces do not carry a fundamental class of pure dimension, various techniques have been developed to construct an ersatz.

Historically, Li--Tian \cite{MR1467172} and Behrend--Fantechi \cite{BF} have proposed solutions to the integration problem by introducing virtual cycles, which allowed cohomological Gromov--Witten invariants to be formally mathematically defined. Using similar techniques, Lee \cite{Lee-QK-2004} constructed a virtual structure sheaf, which is key in defining $K$-theoretical (or in fact $G$-theoretical) invariants. These constructions formalize the objects used by Kontsevich in \cite{Kontsevich95}. The definitions of these virtual objects are not intrinsic; rather, they depend on the choice of a replacement for the cotangent complex of the singular moduli space. The unworkable cotangent complex is replaced locally by a 2-term complex of vector bundles: this is the perfect obstruction theory. For many moduli spaces, the choice of this replacements comes from the geometry of the original moduli problem.

In the seminal paper \cite{Kontsevich95}, Kontsevich proposed a different approach to solve this problem via the notion of differential graded manifolds (or schemes), in short, \emph{$dg$-manifolds}. This idea was developed by Kapranov and Ciocan-Fontanine in \cite{Kapranov-Cio-Font-derived-quot-2001} and\cite{CF-Kapranov-2002-dervied-Hilbert-scheme}.

In \cite{STV}, Sch\"urg-To\"en-Vezzosi use \emph{derived algebraic geometry} to give a more geometric interpretation of these virtual objects. This idea is one of the numerous applications of the field derived algebraic geometry developed by Toën--Vezzosi (eg. \cite{TV05} and \cite{TV08}, see \cite[\S 3.1]{T14} for a nice overview) and by Lurie in \cite{Lurie_SAG}. The derived and \textit{dg} approach are related, but they are not equivalent (see \cite{T14} for the difference).

On the side of differential geometry, Joyce has developed parallel theories of d-manifolds and d-orbifolds and closely related theories of Kuranishi spaces (see \cite{joyce_d_manifolds} for a summary of d-manifolds, \cite{Joyce-Kuranishi-2019} for Kuranishi spaces). Central to the study of moduli spaces are the ideas of derived critical loci \cite{derived_critical_locus}, studied by Vezzosi, and the parallel concept of algebraic d-critical loci introduced by Joyce \cite{d-critical}, as well as those of shifted symplectic structures \cite{shifted_symplectic} of Pantev--To\"{e}n--Vaqui\'{e}--Vezzosi, applied to the study of Donaldson--Thomas invariants by Brav--Bussi--Joyce \cite{darboux-derived}. Nowadays, many works use derived algebraic geometry to study moduli spaces amongst them we recall \cite{MR18} \cite{Kern-QL-2020} \cite{Porta-Yue-Non-archime-2020} \cite{DAG-panorama-2021} \cite{Khan-Virt-class-2019} \cite{Aranha-intrinsic-Artin-2019} \cite{Khan-Aranha-Local-alg-stacks-2022} \cite{Khan-Virtu-excess-inter-2021} \cite{Duff-Joyce-virtual-cycle-2019}, \cite{Joice-Safronov-Lag-2019}. Just as perfect obstruction theories, derived structures on a scheme (or stacks) are not unique: they depend on a choice. In many cases there are natural ones coming from the geometry. \\

{\bf Virtual structure sheaves via derived algebraic geometry.}
In this paper, we use derived algebraic geometry to study the moduli space of sections. In the following we sketch the way in which derived algebraic geometry recovers virtual objects. For a derived stack $\bbR\ffX$, its truncation $t_0(\bbR\ffX)=\ffX$ has a closed embedding or \textit{derived enhancement}:
\[j: \ffX \xhookrightarrow{} \bbR\ffX\text{.}
\]
Informally, $\bbR\ffX$ and $\ffX$ have the same underlying geometric space, but the derived stack is akin to a nilpotent thickening. 
If the derived stack $\bbR\ffX$ is quasi-smooth, that is its cotangent complex is cohomologicaly supported in $(-1,\infty]$, we can define a sheaf class on $\ffX$ via
\[\ccO_{\ffX}^{\vir,\DAG}:= (j_*)^{-1}\ccO_{\bbR\ffX}\text{,}
\]
where $j_{\ast}$ is the induced map between $G$-theory groups, which by dévissage is invertible. 

On the other hand, the derived enhancement gives a perfect obstruction theory for $\ffX$, as long as $\bbR\ffX$ is quasi-smooth and $\ffX$ is a Deligne--Mumford stack. The differential of the inclusion $j$ gives a morphism
\[
dj : j^* \bbL_{\bbR\ffX} \to \bbL_\ffX,
\]
which, under our assumptions, is a perfect obstruction theory \cite[Proposition~1.2]{STV}. Using this perfect obstruction theory, we can follow the recipe of Lee \cite{Lee-QK-2004} to construct a virtual sheaf $\ccO_{\ffX}^{\vir, \mathrm{POT}}$ for $\ffX$. We get an a priori different sheaf on $\ffX$. The equality of $\ccO_{\ffX}^{\vir, \mathrm{POT}}$ and $\ccO_{\ffX}^{\vir, \mathrm{DAG}}$ in the $G$-theory of $\ffX$ is a deep statement, which was proved in \cite[MR, \S 5.4 and \S 5.5]{DAG-panorama-2021} (see also \cite[\S 6]{Porta-Yue-Non-archime-2020}).
\\

{\bf Further directions.}
We believe that our main theorem is part of a new strategy to prove equalities between virtual objects. The strategy is: 
\begin{enumerate}
    \item to construct a morphism at the derived level so that we have a morphism between virtual structure sheaves, and
    \item to prove locally that this morphism is an isomorphism.
    \end{enumerate}
    
For more general statements, one needs to develop a more general machinery: we expect situations in which we have a simple virtual push-forward theorem, but a more complicated relation between derived structure sheaves.

In terms of applications of such a machinery, it is natural to consider stable maps and quasi-maps to a general toric variety $X$ and to try to derive a relation between (derived) structure sheaves. This is not straight-forward, as for a general $X$ there is no morphism
\[
\oc :\bbR\overline{\ccM}_{g,n} (X,d) \dashrightarrow \bbR\overline{\ccQ}_{g,n}(X,d).
\] 
On the other hand, it is possible to get an easy local picture.

We will treat these problems in future works.
\vspace{0.5cm}

{\bf Acknowledgments.}
E. Mann and R. Picciotto received funding by the Agence National de la Recherche (ANR) for the ``Categorification in Algebraic Geometry'' project ANR-17-CE40-0014.
D. Kern received funding from the European Research Council (ERC) under the European Union’s Horizon 2020 research and innovation
programme (Grant ``Derived Symplectic Geometry and Applications'' Agreement No. 768679). C. Manolache was supported by a Royal Society Dorothy Hodgkin Fellowship and an Emmy Noether Fellowship of the London Mathematical Society.

%%%%%%%%%%%%%%%%%%%%%%%%%%%%%%%%%%%%%%%%%%%%%%%ù
\section{Notation}\label{sec:notation}
%%%%%%%%%%%%%%%%%%%%%%%%%%%%%%%%%%%%%%%%%%
\begin{itemize} 
\item Everything is over $\bbC$.
\item For locally free sheaf $\ccE$ on a space $X$ (a scheme, stack, or derived stack) the vector bundle of $\ccE$ is $\bbV(\ccE)\coloneqq {\Spec}_{X}\Sym_{\ccO_X}\ccE^{\vee}$.
\item Let $\ffM$ bean Artin stack with a flat proper family $\pi:\ffC\to\ffM$ of relative dimension $1$. For any morphism $\ffU\to \ffM$, we denote $\pi_{\ffU}:\ffC_{\ffU}\to\ffU$ the pullback of $(\pi,\ffC)$. The most classical example would be $\ffM$ be the moduli of prestable curve, denoted by $\ffM_{g,n}$ of genus $g$ with $n$ marked point and $\ffC_{g,n}$ its universal curve.
\item We use $\bbR$ to mean a derived structure on a geometric object (for example $\bbR X$), and $\rder$ (respectively $\lder$) a right (resp. left) derived functor, for example $\rder f_*$ (resp $\lder f^*$).
\item For $X,Y,Z$ non-derived stacks $\sHom_X(Y,Z)$ are Hom-stacks (relative internal hom) whereas $\Hom_X(Y,Z)$ are groupoids. 
\item For $X,Y,Z$ derived (or non-derived) stacks $\dHom_X(Y,Z)$ are derived Hom-stacks whereas $\bbR\Hom_X(Y,Z)$ are simplicial sets. 
\item For $X$ a non-derived stack, $\ccF$, $\ccG$ sheaves on $X$, $\Hom_{\ccO_X-\mathrm{mod}}(\ccF,\ccG)$ is the global $\Hom$ of sheaves. For $X$ a derived stack and $\ccF$, $\ccG$ complexes of sheaves, $\rder\Hom_{\ccO_X-{dgm}}(\ccF,\ccG)$ denotes the simplicial set associated by the Dold--Kan correspondence to the complex $\Hom^{\bullet}(\ccF,\ccG)$ defined as $\Hom^i(\ccF,\ccG)\coloneqq\Hom^0(\ccF,\ccG[i])$. 
\item $\Picc_{g,n,d}$ (or $\Picc$ for short) is the moduli space of prestable curves of genus $g$ with $n$ marked points together with a degree $d$ line bundle. When we impose some stability conditions, we will write $\Pic$ (see Notation \S \ref{notation}).
\item $\Gw$  and $\RGw$ are the (derived) moduli of  stable maps of genus $g$ with $n$ marked points to projective space $\bbP^r$.
\item  $\Qm$ and $\RQm$ are the (derived) moduli of  quasi-maps of genus $g$ with $n$ marked points to projective space $\bbP^r$.
\item We use the notations $\ffM_{g,n}, \ffC_{g,n}, \Picc,\Pic,\bbR U, \bbR V, W, \UU,...$ for all the objects related to stable maps (for example prestable curve), that is objects where the curve could have rational tails. We put a ``hat'' on the same kind of objects 
$\widehat{\ffM}_{g,n}, \widehat{\ffC}_{g,n}, \widehat{\Picc},\widehat{\Pic},\bbR \widehat{U}, \bbR \widehat{V}, \widehat{W}, \widehat{\UU},...$
for all the objects related to quasi-maps, that is without rational tails.
\end{itemize}

 %%%%%%%%%%%%%%%%%%%%%%%%%%%%%%%%%%%%%%%%%%%%%%%%%

 \section{Derived moduli of sections}\label{Derived-sections}
  %%%%%%%%%%%%%%%%%%%%%%%%%%%%%%%%%%%%%%%%%%%%%%%%%

 In \S \ref{subsec:derived,structure,sections} define a natural derived structure on the moduli of sections (see Corollary \ref{coro:derived,Rsec}) with a compatibility with morphism (cf Proposition \ref{prop:compatibility}), than in \S \ref{ss:cone}, we study the case when $\ffZ$ is a vector bundle and in \S \ref{subsec: tangent,POT,section} we compute the tangent space of the derived moduli space of sections in Theorem \ref{thm:tanS}.
 
 %%%%%%%%%%%%%%%%%%%%%%%%%%%%%%%%%%%%%%%%%%%%%%%%
 \subsection{Derived structure of the moduli of sections}\label{subsec:derived,structure,sections}
 %%%%%%%%%%%%%%%%%%%%%%%%%%%%%%%%%%%%%%%%%%%%%%%

 Let $\ffM$ be a (possibly derived) Artin stack, $\pi:\ffC\to\ffM$ a flat, proper morphism of relative dimension 1. Let $\ffZ$ be a (possibly derived) Artin stack with a smooth morphism $p:\ffZ \to \ffC$.  We have an $\infty$-functor $\pi_*$  called the Weil restriction of scalars, right adjoint to the base-change $\infty$-functor $\pi^*$ (and
 constructed for example in~\cite[Construction 19.1.2.3]{Lurie_SAG}), that will be seen to preserve derived Artin stacks of locally finite presentation as stated  in~\cite{toen2022foliations}:
 \[
 \xymatrix{
 \mathbf{dSt}/\ffM\ar@/^/[r]^{\pi^*}="ladj"& \mathbf{dSt}/\ffC\ar@/^/[l]^{\pi_*}="radj"  \ar@{}"ladj";"radj"|{\perp}.
 }
 \]
  \begin{definition}\label{def:rsec}
 For a derived Artin stack $\ffZ\xrightarrow{p}\ffC$, we denote
 \[
 \bbR\Sec_{\ffM}(\ffZ/\ffC)\coloneqq \pi_*\ffZ.
 \]
 \end{definition}
 Our choice of notation is justified by the following observation.
  \begin{lemma}
   \label{lemma:fctor-pts-sec}
   On any test derived $\ffM$-scheme $T\to\ffM$, writing
   $C_{T}\coloneqq T\times_{\ffM}^{h}\ffC$,
   $Z_{T}\coloneqq C_{T}\times_{\ffC}^{h}\ffZ\simeq T\times_{\ffM}^{h}\ffZ$
   and $p_{T}\colon Z_{T}\to C_{T}$ for the base-changes of $\ffC$,
   $\ffZ$ and $p$, the functor of points of
   $\bbR\Sec_{\ffM}(\ffZ/\ffC)$ evaluates to the spaces of
   $T$-families of sections of $p_{T}$, \emph{i.e.}, intuitively, of
   maps $C_{T}\xrightarrow{s} Z_{T}$ equipped with a homotopy
   $p_{T}\circ s\xRightarrow{\simeq}\id_{C_{T}}$.
 \end{lemma}
 \begin{proof}
   We have from the adjunction defining $\bbR\Sec_{\ffM}(\ffZ/\ffC)$
   that
   \begin{align*}
     \bbR\Sec_{\ffM}(\ffZ/\ffC)(T\to\ffM)
     &=\bbR\Hom_{\ffM}(T,\pi_{\ast}\ffZ)\\
     &\simeq\bbR\Hom_{\ffC}(\pi^{\ast}T,\ffZ)
       =\bbR\Hom_{\ffC}(C_{T},\ffZ)\text{.}
   \end{align*}
   Let now $b\colon C_{T}\to\ffC$ be the canonical morphism. We note
   that $b^{\ast}$, which was already noted to admit a right-adjoint
   $b_{\ast}$, additionally has a ``forgetful'' left-adjoint $b_{!}$,
   given simply by composition along $b$ (see~\cite[paragraph after
   Lemma 6.1.1.1]{Lurie2009}). In particular, we have
   $(C_{T}\xrightarrow{b}\ffC)=b_{!}(C_{T}=C_{T})$, and
   \begin{equation*}
     \bbR\Hom_{\ffC}(C_{T},\ffZ)=\bbR\Hom_{\ffC}(b_{!}C_{T},\ffZ)
     \simeq\bbR\Hom_{C_{T}}(C_{T},b^{\ast}\ffZ)=\bbR\Hom_{C_{T}}(C_{T},Z_{T})\text{.}
   \end{equation*}
   
   Using~\cite[Lemma 5.5.5.12]{Lurie2009} to compute the hom-spaces in
   the slice category
   $\mathbf{dSt}_{/C_{T}}\simeq(\mathbf{dSt}_{/\ffM})_{/C_{T}\to\ffM}$,
   this becomes
  \begin{equation*}
    \bbR\Hom_{C_{T}}(C_{T},Z_{T})
    \simeq\bbR\Hom_{\ffM}(C_{T},Z_{T})
    \times_{\bbR\Hom_{\ffM}(C_{T},C_{T})}^{h}\{\id_{C_{T}\to\ffM}\}\text{,}
  \end{equation*}
  in which we may read the intuitive description.
 \end{proof}
 From this we deduce the following global expression for the moduli
 stack of sections.
 \begin{proposition}
 The derived moduli of sections $\bbR\Sec_{\ffM}(\ffZ/\ffC)$ is  the homotopical cartesian product
 \begin{equation}\label{eq:dsec}
 \begin{tikzcd}
 \bbR\Sec_{\ffM}(\ffZ/\ffC)\arrow[r]\arrow[d] \arrow[dr,phantom,"\ulcorner_h",very near start]& \ffM\ar[d,"i"]\\
 \dHom_{\ffM}(\ffC,\ffZ)\ar[r,"q"] & \dHom_{\ffM}(\ffC,\ffC).
 \end{tikzcd}
 \end{equation}
 where $q$ is induced by composition by $p:\ffZ\to \ffC$ and $i$ is given by the identity morphism. 
 \end{proposition}
 \begin{proof}
  We show that the points of $\bbR\Sec_{\ffM}(\ffZ/\ffC)$ and the
  desired fibre product agree on all test derived
  $\ffM$-schemes. Since representable $\infty$-functors preserve
  (homotopy) limits, we see immediately, using the adjunction which
  defines hom-stacks, that the $T$-points of the fibre product, for
  any $T\to\ffM$, are given by
  \begin{align*}
    &\bbR\Hom_{\ffM}(T,\dHom_{\ffM}(\ffC,\ffZ)\times_{\dHom_{\ffM}(\ffC,\ffC)}^{h}\ffM)\\
    \simeq{}&\bbR\Hom_{\ffM}(T,\dHom_{\ffM}(\ffC,\ffZ))
            \times_{\bbR\Hom_{\ffM}(T,\dHom_{\ffM}(\ffC,\ffC))}^{h}\bbR\Hom_{\ffM}(T,\ffM)\\
    \simeq{}&\bbR\Hom_{\ffM}(T\times_{\ffM}\ffC,\ffZ)
            \times_{\bbR\Hom_{\ffM}(T\times_{\ffM}\ffC,\ffC)}^{h}\bbR\Hom_{\ffM}(T,\ffM)\text{.}
  \end{align*}

  Meanwhile, applying~\cite[Lemma 5.5.5.12]{Lurie2009} this time
  directly to the computation of the hom-space
  $\bbR\Hom_{\ffC}(C_{T},\ffZ)$ in the slice
  $(\mathbf{dSt}_{/\ffM})_{/\ffC\to\ffM}$, we find that
  \begin{equation*}
    \bbR\Sec_{\ffM}(\ffZ/\ffC)(T) \simeq\bbR\Hom_{\ffM}(C_{T},\ffZ)
    \times_{\bbR\Hom_{\ffM}(C_{T},\ffC)}^{h}\ast\text{,}
  \end{equation*}
  which, as $\ffM$ is terminal in $\mathbf{dSt}_{/\ffM}$, agrees with
  the description obtained above.
\end{proof}

\begin{corollary}\label{cor:representability_sec}
  If $\ffZ\to\ffM$ is a locally almost finitely presented (relative)
  $1$-Artin derived stack with quasi-affine diagonal, then
  $\bbR\Sec_{\ffM}(\ffZ/\ffC)\to\ffM$ is a locally almost finitely
  presented $1$-Artin derived stack, with quasi-affine diagonal.
\end{corollary}
\begin{proof}
  Recall that $\pi\colon\ffC\to\ffM$ is assumed flat, so of
  Tor-amplitude $0$. The result is then given by~\cite[Theorem
  5.1.1]{halpern-leistner14:_mappin} for the hom-stacks appearing in
  the fibre product, and algebraicity is stable by base-change.

  Note that this is also \cite[Theorem 19.1.0.1]{Lurie_SAG} in the
  case where $\ffZ\to\ffC$ is a derived algebraic space.
\end{proof}
 
\begin{corollary}\label{coro:derived,Rsec} If $\ffZ,\ffC,\ffM$ are classical (non derived) stacks, the truncation 
 \[
 \Sec_{\ffM}(\ffZ/\ffC)\coloneqq t_0\left( \bbR\Sec_{\ffM}(\ffZ/\ffC)\right)
 \]
 is given by the functor $ \Sec_{\ffM}(\ffZ/\ffC):(\mathbf{Sch}/\ffM)^{op}\to \mathbf{Gpoid}$ taking sections of $\ffZ$ over $\ffC$, that is:
 \begin{align*}
 \Sec_{\ffM}(\ffZ/\ffC)(T\to\ffM)&=\left\{s: C_T\coloneqq T\times_{\ffM}\ffC\to Z_T\coloneqq C_T\times_{\ffC}\ffZ | p_T\circ s=\id_{C_T} \right \}\\
 & =\Hom_{C_T}(C_T,Z_T)\nonumber
 \end{align*}
where $p_T\colon Z_T\to C_T$ is the projection induced by $p$.

 It fits in the fiber product diagram of Artin stacks over $\ffM$:
 \begin{equation}\label{homsq}
   \xymatrix{
     \Sec_{\ffM}(\ffZ/\ffC)\ar[r]\ar[d] \ar@{}[dr]|<<<{\ulcorner}
     & \ffM\ar[d]^{i} \\
     \sHom_{\ffM}(\ffC,\ffZ)\ar[r]^{q} & \sHom_{\ffM}(\ffC,\ffC) 
   }
   \end{equation}
\end{corollary}
\begin{proof}
The functor of points description is obtained by restricting that
  of~\cref{lemma:fctor-pts-sec} to
  $(\mathbf{Sch}/\ffM)\subset(\mathbf{dSch}/\ffM)$.
  
The fibre square follows because the truncation of \cref{eq:dsec} is precisely the diagram in \ref{homsq}, whence we have that
\begin{align*}
t_0( \bbR\Sec_{\ffM}(\ffZ/\ffC))&=t_0( \dHom_{\ffM}(\ffC,\ffZ)\times^h_{\dHom_{\ffM}(\ffC,\ffC)}\ffM)\\
&=\sHom_{\ffM}(\ffC,\ffZ)\times_{\sHom_{\ffM}(\ffC,\ffC)}\ffM\\
&=\Sec_{\ffM}(\ffZ/\ffC).
\end{align*}
\end{proof}

In other words, we have defined a derived enhancement of the moduli space of sections of $\ffZ\to\ffC$ over $\ffM$. We will show in the next section that the tangent complex of this derived enhancement recovers the perfect obstruction theory of the moduli of sections given by \cite{CL12} and much more generally by \cite{W21}.

\begin{definition}[Notation]
Let us use the diagram below to fix the notation for the universal family of the moduli of sections $\bbR\ffS\coloneqq\bbR\Sec_{\ffM}(\ffZ/\ffC)$:
\begin{equation}\label{eq:tower,1}
\begin{tikzcd}
 & \ffZ\ar[d]\\
 \ffC_{\ffS}\ar[r]\ar[ur,"\ev_{\ffS}"] \ar[d,"\pi_{\ffS}"] \ar[dr,phantom,very near start,"\ulcorner"]& \ffC\ar[d,"\pi"]\\
 \ffS\ar[r]& \ffM . 
\end{tikzcd}
\end{equation}
\end{definition}

We prove here a result that will be useful later about compatibility of moduli of sections and their derived structures.
\begin{proposition}\label{prop:compatibility}
Consider 
\begin{equation}\label{eq:tower}
\begin{tikzcd}
\ffZ_2\ar[rr, "q"]\ar[dr, "p_2" ']&&\ffZ_1\ar[dl, "p_1"]\\
&\ffC\ar[d, "\pi"]\\
&\ffM
\end{tikzcd}
\end{equation}
with $\ffZ_1, \ffZ_2$ as in Definition \ref{def:rsec}.
For $i\in\{1,2\}$, we form the moduli of sections
\[
\bbR\ffS_i\coloneqq \bbR\Sec_{\ffM}(\ffZ_i/\ffC).
\]
with their universal curves ${\pi_{\ffS_i}}:\ffC_{\ffS_i}\to\ffS_i$ and evaluations $\ev_{\ffS_i}:\ffC_{\ffS_i}\to\ffZ_i$.
We can also form the moduli of sections over $\ffS_1$ of the pullback of $\ffZ_2$:
\[
\widetilde{\bbR\ffS_2}\coloneqq\bbR\Sec_{\bbR\ffS_1}(\ffC_{\bbR\ffS_1}\times_{\ffZ_1}^h\ffZ_2/\ffC_{\bbR\ffS_1}).
\]
Then
\[
\widetilde{\bbR\ffS_2}\simeq \bbR\ffS_2\times_{\ffM}\bbR\ffS_{1}.
\]
\end{proposition}
\begin{proof}
  For the clarity of this proof (as it is an argument of adjunctions
  and base-change), it will be more convenient to use the algebraic
  notation over the geometric one, that is we write simply
  $\pi_{\ast}$ for the right-adjoint of the base-change functor
  $\pi^{\ast}$: the derived stack of sections is
  $\mathbb{R}\ffS_{i}=\pi_{\ast}(p_{i}\colon\ffZ_{i}/\ffC)$,
  abbreviated to $\pi_{\ast}\ffZ_{i}$, and its universal family is
  written as the base-change
  $\ffC_{\ffS_{i}}=\pi^{\ast}(\pi_{\ast}\ffZ_{i})$. 

  The main actors are summarised in the diagram
  \begin{equation*}
    \begin{tikzcd}
      & \ffZ_{2} \arrow[dr,"q"] \arrow[ddrr,bend left=42,"p_{2}"] & & \\
      q^{\ast}(\ev_{1}\colon\pi^{\ast}\pi_{\ast}\ffZ_{1}/\ffZ_{1})
      \arrow[dr,start anchor=south east,"\ev_{1}^{\ast}q"']
      \arrow[ur,start anchor=north east] \arrow[rr,phantom,very
      near start,"<^{h}"] & & \ffZ_{1} \arrow[dr,"p_{1}"] & \\
      & \pi^{\ast}\pi_{\ast}\ffZ_{1} \arrow[d,"a^{\ast}\pi"']
      \arrow[rr,"\pi^{\ast}a"'] \arrow[ur,"\ev_{1}"]
      \arrow[drr,phantom,very near start,"\ulcorner_h"] & & \ffC
      \arrow[d,"\pi"] \\
      & \pi_{\ast}\ffZ_{1} \arrow[rr,"a"'] & & \ffM
    \end{tikzcd}
  \end{equation*}
  where in particular we read (by comparing the formulas) that the
  stack $\widetilde{\bbR\ffS_2}$ is in this algebraic notation
  \begin{equation*} (a^{\ast}\pi)_{\ast}\bigl(\ev_{1}^{\ast}q\colon
    q^{\ast}(\ev_{1}\colon\pi^{\ast}\pi_{\ast}\ffZ_{1}/\ffZ_{1})
    /\pi^{\ast}\pi_{\ast}\ffZ_{1}\bigr)\text{.}
  \end{equation*}

  Since $p_{1}\circ\ev_{1}=\pi^{\ast}a$, the top cartesian square glues
  with the corresponding triangle to provide an equally cartesian
  square, which by symmetry of base-change we see as exhibiting an
  equivalence
  \begin{equation*}
    q^{\ast}(\ev_{1}\colon\pi^{\ast}\pi_{\ast}\ffZ_{1}/\ffZ_{1})
    =\ev_{1}^{\ast}(q\colon\ffZ_{2}/\ffZ_{1})
    \simeq(\pi^{\ast}a)^{\ast}(p_{2}\colon\ffZ_{2}/\ffC)
  \end{equation*}
  of the top-left fibre product with
  $\ffZ_{2}\times_{\ffM}\pi_{\ast}\ffZ_{1}$, so that now
  \begin{equation*}
    \widetilde{\bbR\ffS_2}
    \simeq(a^{\ast}\pi)_{\ast}(\pi^{\ast}a)^{\ast}(p_{2}\colon\ffZ_{2}/\ffC)\text{.}
  \end{equation*}

  As the arrow category (``target'') bicartesian fibration --- giving the
  adjunctions $\pi_{!}\dashv\pi^{\ast}$ --- satisfies the Beck--Chevalley
  condition on pullback squares (\cite[Proposition 4.2.5]{weinberger22:beck-chevalley-fibrations} shows that it is a Beck--Chevalley fibration, and \cite[Proposition 3.2]{weinberger22:beck-chevalley-fibrations} combined with~\cite[Corollary 2.4.7.12]{Lurie2009} relates this to the ``base-change'' form of the Beck--Chevalley condition, as also noted in~\cite[Remark 4.1.5]{hopkins-lurie13:ambidextry}), we have
  in the lower cartesian square that
  $\pi^{\ast}a_{!}\simeq(\pi_{\ast}a)_{!}(a^{\ast}\pi)^{\ast}$, which
  passing to right-adjoints gives an equivalence
  $a^{\ast}\pi_{\ast} \simeq(a^{\ast}\pi)_{\ast}(\pi^{\ast}a)^{\ast}$.

  Putting everything together, we finally get that
  \begin{equation*}
    \widetilde{\bbR\ffS_2}
    \simeq(a^{\ast}\pi)_{\ast}(\pi^{\ast}a)^{\ast}(p_{2}\colon\ffZ_{2}/\ffC)
    \simeq a^{\ast}\pi_{\ast}(p_{2}\colon\ffZ_{2}/\ffC)
  \end{equation*}
  which is the desired result.

\end{proof}

 \begin{example}[Moduli of stable maps]\label{ex:stablemaps}
 Let $\ffC\xrightarrow{\pi}\ffM$ be the moduli space of pre-stable genus $g$, $n$-pointed curves with its universal family. Let $\ffZ=\ffC\times X$ for a smooth projective variety $X$. Then 
 \[
 \Sec_{\ffM}(\ffC\times X/\ffC)=\sHom_{\ffM}(\ffC,\ffM\times X).
 \]
 For any choice of effective class $\beta$, the moduli space $\overline{\ccM} (X,\beta)$ of stable maps to $X$ is then an open substack of the moduli of sections $\Sec_{\ffM}(\ffC\times X/\ffC)$.
 Similarly, 
 \[
 \bbR\Sec_{\ffM}(\ffC\times X/\ffC)=\dHom_{\ffM}(\ffC,\ffM\times X).
 \]
 The usual derived enhancement of the moduli of stable maps \cite[Section 2]{STV}, denoted $\bbR\overline{\ccM} (X,\beta)$, is the unique derived structure on $\overline{\ccM} (X,\beta)$ which makes the following diagram homotopy Cartesian
\[
\xymatrix{
\overline{\ccM} (X,\beta)\ar@{^{(}->}[r]\ar@{^{(}->}[d] & \bbR\overline{\ccM} (X,\beta)\ar@{^{(}->}[d]\\
 \Sec_{\ffM}(\ffC\times X/\ffC) \ar@{^{(}->}[r]& \bbR\Sec_{\ffM}(\ffC\times X/\ffC).
}
\]
\end{example}
\begin{remark}
 The stack $\bbR\Sec_{\ffM}(\ffZ/\ffC)$ is in general a derived stack even if $(\ffM,\ffC,\ffZ)$ is a triple of classical stacks. 
\end{remark}
 
%%%%%%%%%%%%%%%%%%%%%%%%%%%%%%%%%%%%%%%%%%%%%%%%%%
\subsection{The linear case: $\ffZ$ is a vector bundle }\label{ss:cone}
%%%%%%%%%%%%%%%%%%%%%%%%%%%%%%%%%%%%%%%%%%%%%%%%%%%%%

We will now consider the important special case where $\ffZ$ is a linear stack. As we will see, this case covers many disparate constructions: moduli of stable maps to projective spaces (see \cref{stable_maps_as_sections}), and more generally to varieties which are GIT quotients by linear groups, as well as moduli spaces of quasi-maps (see \cref{ss:qm}) and moduli of stable maps with fields (\cref{ex:derived_fields}). 
In this case, the derived moduli of sections is an affine stack over its base.

We start with a review of the classical (non derived) construction. Let $\ffZ=\bbV(\ccE):=\Spec_\ffC \Sym(\ccE^\vee)$ for $\ccE$ a locally-free sheaf over $\ffC$. As proved in \cite{CL12}, sections of $\bbV(\ccE)$ over $\ffM$ are an affine scheme, in fact an abelian cone:
\[
\Sec_{\ffM}(\bbV(\ccE)/\ffC)=\Spec_{\ffM}\Sym(\rder^1\pi_*\ccE^{\vee}\otimes\omega_{\pi})\text{.}
\]
 Indeed, let $f:T\to\ffM$ and $\widehat{f}:C_T\to\ffC$, by Serre's duality and flat base change we have
 \begin{align*}
 \Sec_{\ffM}(\bbV(\ccE)/\ffC)(T\to\ffM)&=\Hom_{C_T}(C_T,\widehat{f}^*\ccE)\\
 &=\Hom_{\ccO_T-\mathrm{mod}}(\rder^1{\pi_{T*}}\widehat{f}^*\ccE^{\vee}\otimes\omega_{\pi_T},\ccO_T)\\
 &= \Hom_{\ccO_T-\mathrm{mod}}(\rder^1{\pi_{T*}}\widehat{f}^*(\ccE^{\vee}\otimes\omega_{\pi}),\ccO_T)\\
 &=\Hom_{\ccO_T-\mathrm{mod}}(f^*\rder^1{\pi_*}\ccE^{\vee}\otimes\omega_{\pi},\ccO_T)\\
 &= \Spec_{\ffM}\Sym(\rder^1\pi_*\ccE^{\vee}\otimes\omega_{\pi})(T\to\ffM)\text{.}
 \end{align*}
 \begin{example}[Hodge bundle]\label{ex:Hodge}
 For $\ffM=\ffM^{\pre}_{g,n}$, the moduli of pre-stable curves, the Hodge bundle $\ffH$ is the cone of sections 
 \[\Sec_{\ffM}(\bbV(\omega_{\pi})/\ffC)=\Spec_{\ffM}\Sym(\rder^1\pi_*\ccO_{\ffC})\text{.}\]
 This is a vector bundle of rank $g$, since $\rder^0\pi_*\ccO_{\ffC}\cong\ccO_{\ffM}$.
 \end{example}
\begin{example}[Stable maps with fields]\label{ex:fields}
Let $\ffX=\overline{\ccM}_{g,n} (X,\beta)$ with its universal family  
\[
(\pi_{\ffX},\ev_{\ffX}):\ffC_{\ffX}\to \ffX\times X.
\]
Let $\ccE$ be a locally-free sheaf over $X$. The moduli space of stable maps with fields in $E=\bbV(\ccE)\to X$ (see \cite{CL12, QJW21, picciotto2021moduli}) denoted $\ffX^E$ can be seen as 
\[
\ffX^E=\Sec_{\ffX}(\bbV(\ev_{\ffX}^*\ccE^{\vee}\otimes\omega_{\ffC_{\ffX}/\ffX})/\ffC_{\ffX})=\Spec_{\ffX}\Sym (\rder^1\pi_{\ffX *}\ev_{\ffX}^*\ccE)\text{.}
\]
\end{example}

We will now cover the general case where $\ffZ$ is a derived vector bundle, that is 
\[
\ffZ=\bbV(\ccE)\coloneqq \bbR\Spec_{\ffC}\left(\Sym (\ccE^{\vee})\right)
\]
for $\ccE\in \Perf^{\geq 0}(\ccO_{\ffC})$.
\begin{proposition}\label{prop:moduli_section_vector_bundle}
Let $\ffZ=\bbV(\ccE)$ for $\ccE$ as above. Then
\[
\bbR\Sec_{\ffM}(\bbV(\ccE)/\ffC)=\bbR\Spec_{\ffM}\Sym((\rder\pi_{\ffX *}\ccE)^{\vee})\text{.}
\]
\end{proposition}
\begin{proof}
This is formally similar to the argument in \cref{ss:cone}. Let $f:T=\bbR\Spec A_{\bullet}\to\ffM$ be an affine derived scheme over $\ffM$ with $\widehat{f}:C_T\coloneqq T\times^h_{\ffM}\ffC\to\ffC$ the induced map and $\pi_T:C_T\to T$ the induced projection. From \cref{def:rsec}, we have
\begin{align*}
 \bbR\Sec_{\ffM}(\bbV(\ccE)/\ffC)(T\to\ffM)&= \bbR\Hom_{T}(C_T,\bbV(\lder\widehat{f}^*\ccE))\times^h_{\bbR\Hom_T(C_T,C_T)}T\\
 & = \bbR\Hom_{C_T}(C_T,\bbV(\lder\widehat{f}^*\ccE))\\
 & = \rder\Hom_{\ccO_{C_T}-\mathrm{dgm}}(\ccO_{C_T},\lder\widehat{f}^*\ccE)\text{.}
\end{align*}
The second line follows by \cite[5.5.5.12]{Lurie2009}.
On the other hand,
\begin{align*}
\bbR\Spec_{\ffX}\Sym((\rder\pi_{\ffX *}\ccE)^{\vee})(T\to\ffM)&=\rder\Hom_{\ccO_{\ffM}-\mathrm{cdga}}(\Sym((\rder\pi_{\ffX *}\ccE)^{\vee},\rder f_*\ccO_T)\\
&=\rder\Hom_{\ccO_{T}-\mathrm{dgm}}((\lder f^*R\pi_{\ffX *}\ccE)^{\vee},\ccO_T)\text{.}
\end{align*}
By flat base-change, 
\[
\rder\Hom_{\ccO_{T}-\mathrm{dgm}}((\lder f^*\rder\pi_{\ffX *}\ccE)^{\vee},\ccO_T) = \rder\Hom_{\ccO_{T}-\mathrm{dgm}}((\rder\pi_{T *}\lder\widehat{f}^*\ccE)^{\vee},\ccO_T)\text{.}
\]
By the sheafified Grothendieck duality statement of \cite[Corollary 4.4.2]{MR2681711},
\begin{align*}
 \rder\Hom_{\ccO_{T}-\mathrm{dgm}}((\rder\pi_{T *}\lder\widehat{f}^*\ccE)^{\vee},\ccO_T) &= \rder\Hom_{\ccO_{T}-\mathrm{dgm}}(\rder\pi_{T *}\rder\ccH\mathit{om}_{C_T}(\lder\widehat{f}^*\ccE,\omega_{\pi_T}),\ccO_T)\\
 &= \rder\Hom_{\ccO_{T}-\mathrm{dgm}}(\rder\pi_{T *}(\lder\widehat{f}^*\ccE^{\vee}\otimes\omega_{\pi_T}),\ccO_T)\text{.}
\end{align*}
By the global duality statement of \cite[Theorem 4.1.1]{MR2681711},
\[
 \rder\Hom_{\ccO_{T}-\mathrm{dgm}}(\rder\pi_{T *}(\lder\widehat{f}^*\ccE^{\vee}\otimes\omega_{\pi_T}),\ccO_T)=  \rder\Hom_{\ccO_{C_T}-\mathrm{dgm}}(\lder\widehat{f}^*\ccE^{\vee}\otimes\omega_{\pi_T},\pi_{T}^{!}\ccO_T)\text{.}
 \]
 So finally,
 \begin{align*}
 \rder\Hom_{\ccO_{C_T}-\mathrm{dgm}}(\lder\widehat{f}^*\ccE^{\vee}\otimes\omega_{\pi_T},\pi_{T}^{!}\ccO_T) & = \rder\Hom_{\ccO_{C_T}-\mathrm{dgm}}(\lder\widehat{f}^*\ccE^{\vee}\otimes\omega_{\pi_T},\omega_{\pi_T})\\
&= \rder\Hom_{\ccO_{C_T}-\mathrm{dgm}}(\ccO_{C_T},\lder\widehat{f}^*\ccE)\text{.}
 \end{align*}

\end{proof}
\begin{example}[Derived Hodge bundle]\label{ex:RH}
The derived version of the Hodge bundle of Example \ref{ex:Hodge} is 
\[
\bbR\ffH=\bbR \Sec_{\ffM}(\bbV(\omega_{\pi})/\ffC)=\ffH\times_{\ffM}\underline{\bbA^1}_{\ffM}[-1].
\]
This consists of the usual Hodge bundle in degree 0 and a trivial line bundle in degree 1. 
\end{example}

\begin{example}[Derived stable maps with fields]\label{ex:derived_fields}
Keeping the notation from \cref{ex:fields}, we define the derived version of the moduli space of stable maps with fields. We have from \cref{ex:stablemaps} a derived enhancement of the moduli of stable maps to $X$, $\bbR\ffX\coloneqq\bbR\overline{\ccM}_{g,n}$ with a universal family $\pi_{\bbR\ffX}, \ev_{\bbR\ffX}:\ffC_{\bbR\ffX }\to\bbR\ffX\times X$. The derived enhancement of the moduli of stable maps can be constructed as
\[\bbR\ffX^E=\bbR\Sec_{\bbR\ffX}(\bbV(ev^*_{\bbR\ffX}\ccE^{\vee}\otimes\omega_{\ffC_{\bbR\ffX}/\bbR\ffX})/\ffC_{\bbR\ffX})=\bbR\Spec_{\bbR\ffX}\Sym(\rder\pi_{\bbR\ffX_*}ev_{\bbR\ffX}^*\ccE[1])\text{,}
\]
the second equality coming from \cref{prop:moduli_section_vector_bundle} and Grothendieck duality.
\end{example}

The derived definition of the moduli of sections restricts well, in the sense of the following proposition.
\begin{proposition}\label{prop:restriction of sec}
Let $\ccE$ be a derived vector bundle over $\ffC$. Let $\ffU\subset\ffM$ be an open substack. Denote $\ffC_{\ffU}=\ffC\times_{\ffM}\ffU$ and $\ccE_{\ffU}$ the restriction of $\ccE$. Then
\[
\bbR\Sec_{\ffU}(\bbV(\ccE_{\ffU})/\ffC_{\ffU})\simeq \ffU\times_{\ffM}\bbR\Sec_{\ffM}(\bbV(\ccE)/\ffC)\text{.}
\]
In particular $\bbR\Sec_{\ffU}(\bbV(\ccE_{\ffU})/\ffC_{\ffU})$ is an open derived substack of $\bbR\Sec_{\ffM}(\bbV(\ccE)/\ffC)$.
\end{proposition}
\begin{proof}
Fix the following notation:
\[
\begin{tikzcd}
\ffC_{\ffU}\ar[r,"\overline{i}"]\ar[d,"\pi_{\ffU}" '] \ar[dr,phantom,"\ulcorner", very near start] & \ffC\ar[d,"\pi"]\\
\ffU\ar[r,"i"] & \ffM\text{.}
\end{tikzcd}
\]
Then
\begin{align*}
\bbR\Sec_{\ffU}(\bbV(\ccE_{\ffU})/\ffC_{\ffU}) & =\bbR\Spec_{\ffU}\Sym(\rder\pi_{\ffU *}\overline{i}^*\ccE)\\
& \simeq\bbR\Spec_{\ffU}\Sym(i^* \rder\pi_*\ccE)\\
& = \bbR\Spec_{\ffU}(i^*\Sym(\rder\pi_*\ccE))\\
& = \ffU\times_{\ffM}\bbR\Spec_{\ffM}\Sym(\rder\pi_*\ccE)\\
& = \ffU\times_{\ffM}\bbR\Sec_{\ffM}(\bbV(\ccE)/\ffC)\text{.}
\end{align*}
Note that from the first to the second line we have used the flatness of $i$.
\end{proof}
\begin{remark}
The above proof goes through unchanged for $\ffU\to\ffM$ a flat morphism. Then $\bbR\Sec_{\ffU}(\bbV(\ccE_{\ffU})/\ffC_{\ffU})\simeq \ffU\times_{\ffM}\bbR\Sec_{\ffM}(\bbV(\ccE)/\ffC)\to \bbR\Sec_{\ffM}(\bbV(\ccE)/\ffC)$ is a flat morphism.
\end{remark}

 %%%%%%%%%%%%%%%%%%%%%%%%%%%%%%%%%%%%%%%%%%%%%%%%%%
\subsection{Tangent complex and perfect obstruction theory}\label{subsec: tangent,POT,section}
 %%%%%%%%%%%%%%%%%%%%%%%%%%%%%%%%%%%%%%%%%%%%%%%%%%

Recall that for a derived Hom-stack $H\coloneqq\dHom_X(Y,Z)$ we have a universal family
\[
\xymatrix{
H\times^h_X Y\ar[r]^{\qquad\ev_H}\ar[d]^{\pi_H}& Z\\
H& 
}
\]
and the relative tangent complex $\bbT_{H/X}$ is given by the following simple expression (see \cite[Thm 5.4.8]{CF-Kapranov-2002-dervied-Hilbert-scheme} or \cite[p.13]{STV} or the proof of \cite[Prop.4.3.1]{MR18}):

\begin{equation}\label{tanH}
\bbT_{H/X}=\rder\pi_{H*}\lder\ev_{H}^{*}\bbT_{Z/X}\text{.}
\end{equation}

Applying this fact to the diagram in \cref{def:rsec} allows us to compute $\bbT_{\bbR\Sec_{\ffM}(\ffZ/\ffC)/\ffM}$. The cotangent complex of a Weil restriction is also computed in \cite[\S 19.1.4]{Lurie_SAG}.
\begin{theorem}\label{thm:tanS}
Let $\bbR\ffS\coloneqq \bbR\Sec_{\ffM}(\ffZ/\ffC)$, as per our convention we have $\pi_{\bbR\ffS}\colon\ffC_{\bbR\ffS}=\bbR\ffS\times^h_{\ffM}\ffC\to \bbR\ffS$ and $\ev_{\bbR\ffS}\colon\ffC_{\bbR\ffS}\to\ffZ$.
\[
\bbT_{\bbR\Sec_{\ffM}(\ffZ/\ffC)/\ffM}=\rder\pi_{\bbR\ffS *}\lder \ev_{\bbR\ffS}^*\bbT_{\ffZ/\ffC}\text{.}
\]
\end{theorem}
\begin{proof}
For simplicity, we adopt the following notation
\begin{equation}\label{square1}
 \xymatrix{
\bbR\ffS= \bbR\Sec_{\ffM}(\ffZ/\ffC)\ar[r]^k\ar[d]^l & \ffM\ar[d]^{i}\\
 H\coloneqq\dHom_{\ffM}(\ffC,\ffZ)\ar[r]^q \ar[dr]& K\coloneqq \dHom_{\ffM}(\ffC,\ffC)\ar[d]\\
 & \ffM
 }
\end{equation}
and all push-forwards and pull-backs are understood as derived in this proof. We also have the following commutative diagram of evaluation maps and projections
\begin{equation}\label{square2}
\xymatrix{
 & & & \ffZ\ar[d]^p\\
 & & & \ffC\\
 \ffC_{\bbR{\ffS}}\ar[r]^{\widehat{l}}\ar[uurrr]^{\ev_{\bbR\ffS}}\ar[d]^{\pi_{\bbR\ffS}}& \ffC_{H} \ar[r]^{\widehat{q}}\ar[uurr]_{\ev_H}\ar[d]^{\pi_H}& \ffC_{K} \ar[ur]_{\ev_K}\ar[r]\ar[d]^{\pi_K} & \ffC\ar[d]\\
\bbR\ffS\ar[r]^{l}& H\ar[r]^q& K\ar[r] & \ffM\text{.}
}
\end{equation}
The bottom squares are all fibered.
The rightmost vertical maps of \cref{square1} give
\[
\bbT_{\ffM/K}=i^*\bbT_{K/\ffM}[-1]=i^*(\pi_{K*}\ev_{K}^*\bbT_{\ffC/\ffM})[-1]
\]
because the composite vertical map is the identity and by \cref{tanH}. Then we have 
\[
\bbT_{\bbR\ffS/H}=k^*\bbT_{\ffM/K}=k^*i^*(\pi_{K*}\ev_{K}^*\bbT_{\ffC/\ffM})[-1]\text{.}
\]
We can then retrieve $\bbT_{\bbR\ffS/\ffM}$ from the following rotation of the distinguished triangle of $\bbR\ffS\to H\to\ffM$:
\begin{equation}\label{triangle1}
\bbT_{\bbR\ffS/\ffM}\to l^*\bbT_{H/\ffM}\to \bbT_{\bbR\ffS/H}[1]
\end{equation}
by knowing that $l^*\bbT_{H/\ffM}=l^*\pi_{H *}\ev_{H}^*\bbT_{\ffZ/\ffM}$. To do this we first check that by diagram \ref{square2} we have the following natural transformations
\begin{align*}
&k^*i^*\pi_{K*}\ev^*_K=l^*q^*\pi_{K*}\ev^*_K=l^*\pi_{H*}\widehat{q}^*ev_K^*=\pi_{\bbR\ffS *}\widehat{l}^*\widehat{q}^*ev_K^*=\pi_{\bbR\ffS *}\ev_{\bbR\ffS }^*p^*\\
&l^*\pi_{H *}\ev_{H}^*=\pi_{\bbR\ffS *}\widehat{l}^*\ev_H^*=\pi_{\bbR\ffS *}\ev_{\bbR\ffS}^*\text{.}
\end{align*}
\[
\bbT_{\bbR\ffS/\ffM}\to\pi_{\bbR\ffS *}\ev_{\bbR\ffS }^*\bbT_{\ffZ/\ffM}\to \pi_{\bbR\ffS *}\ev_{\bbR\ffS }^*p^*\bbT_{\ffC/\ffM}\text{.}
\]
This triangle is obtained simply by applying $\pi_{\bbR\ffS *}\ev_{\bbR\ffS }^*$ to $\bbT_{\ffZ/\ffC}\to\bbT_{\ffZ/\ffM}\to T_{\ffC/\ffM}=\bbT_{\ffC/\ffM}$, so
\[
\bbT_{\bbR\ffS/\ffM}=\pi_{\bbR\ffS *}\ev_{\bbR\ffS }^*\bbT_{\ffZ/\ffC}\text{.}
\]
\end{proof}
\begin{corollary}\label{cor:pot}
If $\ffZ\to\ffC$ is a smooth Deligne--Mumford stack, 
\[\ffS\coloneqq\Sec_{\ffM}(\ffZ/\ffC)=t_0(\bbR\Sec_{\ffM}(\ffZ/\ffC))
\]
has a relative perfect obstruction theory in the sense of \cite{BF} given by
\[
\bbT_{\ffS/\ffM}\to \bbE_{\ffS/\ffM}\coloneqq \rder\pi_{\ffS *} \ev_{\ffS}^*T_{\ffZ/\ffC}\text{.}
\]
\end{corollary}
\begin{proof} In this case the derived stack $\bbR\ffS$ is quasi-smooth and $\ffS$ is its truncation. Let  $\ffS\xhookrightarrow{j}\bbR\ffS$ be  the closed immersion of the non-derived moduli of sections given by truncation. Then $\bbT_{\ffS/\ffM}\to j^*\rder\pi_{\bbR\ffS *} \ev_{\bbR\ffS}^*T_{\ffZ/\ffX}$ is a perfect obstruction theory by \cite[Corollary 1.3]{STV}.  There is a quasi-isomorphism 
\[ j^*\rder\pi_{\bbR\ffS *} \ev_{\bbR\ffS}^*T_{\ffZ/\ffX}= \rder\pi_{\ffS *} \ev_{\ffS}^*T_{\ffZ/\ffC}
\]
 given by base-change of the projection and evaluation maps, hence the result.
\end{proof}

\begin{example}[Pic parametrizing prestable curves with a line bundle]\label{ex:pic}
Let $\ffM=\ffM_{g,n}^{\pre}$ be the moduli space of pre-stable genus $g$, $n$-pointed curves with universal curve $\ffC$. Consider the moduli space $\Pic_d:=\Pic_{g,n,d}$ parametrizing pairs $(C,L)$ of a pre-stable curve and a line bundle of degree $d$ with the additional ``stability'' conditions
\begin{enumerate}
\item \[\omega_{C}^{\log}\otimes L^{\otimes 3}>0\] where $\omega_C^{\log}$ is the canonical bundle of the curve twisted by the sum of the $n$ marked points.
\item \[ \deg(L)|_{C_i}\geq 0\] on all components $C_i$ of $C$.
\end{enumerate}
This is an open substack of the usual stack of curves with a degree $d$ line bundle, denoted by $\Pic_d$. It turns out that $\Pic_d$ is an open and closed substack of the moduli of sections of $\ffC\times B\bbC^*$ that is
\[
\Pic_d\subset \bbR\Sec_{\ffM}(\ffC\times B\bbC^*/\ffC)\text{.}
\]
The pullback of the universal curve over $\Pic_d$ is denoted as usual by $\pi_{\Pic_d}:\ffC_{\Pic_d}\to\Pic_d$. The universal section induces an evaluation $\ell_d:\ffC_{\Pic_d}\to B\bbC^*$.
By Theorem \ref{thm:tanS}, the relative tangent of the morphism $\Pic_d\to\ffM$ is
\begin{align*}
\bbT_{\Pic_d/\ffM} &= \rder\pi_{\Pic_d *}\lder\ell_d^*\bbT_{B\bbC^*}\\
& = \rder\pi_{\Pic_d *}\ccO_{\ffC_{\Pic_d}}[1]\text{.}
\end{align*}
$\Pic_d\to\ffM$ is a smooth Artin stack of relative dimension $g-1$.
\end{example}

%%%%%%%%%%%%%%%%%%%%%%%%%%%%%%%%%%%%%%%%%%%%%%%%%%%%%
\section{Derived structure on stable maps and quasi-maps}\label{sec:compare_derived_stable_maps}
%%%%%%%%%%%%%%%%%%%%%%%%%%%%%%%%%%%%%%%%%%%%%%%%%%%%%
There are several ways of constructing derived moduli spaces of maps to a quotient.  The rest of the paper is concerned with maps to projective space $\bbP^r$. Below we describe the construction of stable maps and quasi-maps to projective spaces as particular cases of the moduli space of sections.
\subsection{Stable maps to $\bbP^r$ as sections} \label{stable_maps_as_sections}
From example \ref{ex:stablemaps}, we can construct $\RGw$ as an open substack of $\bbR\Sec_{\bbM}(\ffC\times\bbP^r/\ffC)$. Then theorem \ref{thm:tanS} recovers the usual formula
\[
\bbT_{\RGw/\ffM}=\rder\pi_{\RGw *}f^*T_{\bbP^r}
\]
where $f:\ffC_{\RGw}\to\bbP^r$ is the second component of the universal evaluation $\ev_{\RGw}$.

On the other hand, we may view degree $d$ maps into $\bbP^r$ as (an open substack of) sections of $r+1$ degree $d$ line bundles over a curve. With notation from example \ref{ex:pic} we can define the universal bundle of  $\Pic_d$ as the pullback of the universal bundle $[\bbC/\bbC^*]$ over the classifying space $B\bbC^*=[\bullet/\bbC^*]$:
\[
\xymatrix{
\ffL_d\ar[r]\ar[d]& [\bbC/\bbC^*]\ar[d]\\
\ffC_{\Pic_d}\ar[r]^{\ell_d} & [\bullet/\bbC^*]\text{.}
}
\]
The corresponding locally-free sheaf is denoted by $\ccL_d$.
 In the non-derived setting, this is indeed well-known that we can think of stable maps to projective space as an open substack of the moduli of sections of line bundles:
\[
\Gw\subset\Sec_{\Pic_d}(\ffL^{\oplus r+1}/\ffC_{\Pic_d})\text{.}
\]
This description gives rise to a perfect obstruction theory relative to $\Pic_d$ which has been proved to be compatible with the usual one (see for example \cite{CL12,Fontanine-QMap2010}).
In the discussion below, we strengthen previous results by proving a derived statement (our \cref{thm:der_structures_on_Mgn}) which easily implies the classical one (\cref{thm:potMgn}).

Observe that $\ffL_d^{\oplus r+1}=\ffC_{\Pic_d}\times_{[\bullet/\bbC^*]}[\bbC^{r+1}/\bbC^*]$. The projective space $\bbP^r$  is open in the global quotient stack $[\bbC^{r+1}/\bbC^*]$, so at the level of derived moduli of sections we obtain an open immersion:
\begin{equation}\label{eq:Prquotient}
\bbR\Sec_{\Pic_d}(\ffC_{\Pic_d}\times_{[\bullet/\bbC^*]}\bbP^r/\ffC_{\Pic_d}) \subset\bbR\Sec_{\Pic_d}(\ffL_d^{\oplus r+1}/\ffC_{\Pic_d})\text{.}
\end{equation}

\begin{claim}\label{claim:secs}
We can identify the derived stacks of sections  $\bbR\Sec_{\Pic_d}(\ffC_{\Pic_d}\times_{[\bullet/\bbC^*]}\bbP^r/\ffC_{\Pic_d})$ and $\bbR\Sec_{\ffM}(\ffC\times\bbP^r/\ffC)$.
\end{claim}
This gives us two ways of obtaining a derived enhancement of $\Gw$:
\begin{itemize}
    \item From the open immersion $\Gw\subset \Sec_{\ffM}(\ffC\times\bbP^r/\ffC)$ and the enhancement $\Sec_{\ffM}(\ffC\times\bbP^r/\ffC)\xhookrightarrow{} \bbR\Sec_{\ffM}(\ffC\times\bbP^r/\ffC)$, or
    \item From the open immersion $\Gw\subset \Sec_{\Pic_d}(\ffL_d^{\oplus r+1}/\ffC_{\Pic_d})$ and the derived enhancement $\Sec_{\Pic_d}(\ffL_d^{\oplus r+1}/\ffC_{\Pic_d})\xhookrightarrow{} \bbR\Sec_{\Pic_d}(\ffL^{\oplus r+1}/\ffC_{\Pic_d})$.
\end{itemize}
We will see below that these two enhancements are equivalent, thus we use the notation $\RGw$ freely for either of them. 

Indeed, observe that the open immersion $\Gw\subset \Sec_{\Pic_d}(\ffL_d^{\oplus r+1}/\ffC_{\Pic_d})$ factors as
\[
 \Gw \subset \Sec_{\Pic_d}(\ffC_{\Pic_d}\times_{[\bullet/\bbC^*]}\bbP^r/\ffC_{\Pic_d})\subset\Sec_{\Pic_d}(\ffL^{\oplus r+1}/\ffC_{\Pic}) 
 \]
and all the derived enhancements of the middle space coincide. The enhancement $\Sec_{\Pic_d}(\ffC_{\Pic_d}\times_{[\bullet/\bbC^*]}\bbP^r/\ffC_{\Pic_d})\xhookrightarrow{} \bbR\Sec_{\Pic_d}(\ffC_{\Pic_d}\times_{[\bullet/\bbC^*]}\bbP^r/\ffC_{\Pic_d})$  coming from its own $\Sec$ description is equivalent to the one coming from the second inclusion and the enhancement $\Sec_{\Pic_d}(\ffL^{\oplus r+1}/\ffC_{\Pic}) \xhookrightarrow{} \bbR\Sec_{\Pic_d}(\ffL^{\oplus r+1}/\ffC_{\Pic}) $ by \cref{eq:Prquotient}. The enhancement coming from the identification 
\[
\Sec_{\Pic_d}(\ffC_{\Pic_d}\times_{[\bullet/\bbC^*]}\bbP^r/\ffC_{\Pic_d})\simeq \Sec_{\ffM}(\ffC\times\bbP^r/\ffC)\xhookrightarrow{} \bbR\Sec_{\ffM}(\ffC\times\bbP^r/\ffC)
\]
is equivalent to the enhancement $
\Sec_{\Pic_d}(\ffC_{\Pic_d}\times_{[\bullet/\bbC^*]}\bbP^r/\ffC_{\Pic_d})\xhookrightarrow{} \bbR\Sec_{\Pic_d}(\ffC_{\Pic_d}\times_{[\bullet/\bbC^*]}\bbP^r/\ffC_{\Pic_d})$ by Claim \ref{claim:secs}.
This shows that there is a unique derived enhancement $\RGw$ of $\Gw$ that is open in both $\bbR\Sec_{\ffM}(\ffC\times\bbP^r/\ffC)$ and $\bbR\Sec_{\Pic_d}(\ffL^{\oplus r+1}/\ffC_{\Pic})$.
\begin{proof}[Proof of Claim \ref{claim:secs}]
The claim follows from Proposition \ref{prop:compatibility} with $\ffZ_1=[\bullet/\bbC^*]\times\ffC$ and $\ffZ_2=\bbP^r\times\ffC$, since the parallelogram below is Cartesian: 
\[
\xymatrix{
& \bbP^r\times\ffC\ar[d]\\
\bbP^r\times_{[\bullet/\bbC^*]}\ffC_{\Pic}\ar[ur]\ar[d] & [\bullet/\bbC^*]\times\ffC \ar[d]\\
\ffC_{\Pic}\ar[r]\ar[ur]_{\ev_{\Pic_d}}& \ffC \text{.}
}
\]
\end{proof}

We summarize the discussion of this subsection:
\begin{theorem}\label{thm:der_structures_on_Mgn}
The derived enhancement $\RGw$ of $\Gw$ defined in \cref{ex:stablemaps} from the inclusion $\Gw\subset\Sec_{\ffM}(\ffC\times\bbP^r/\ffC)$ is a derived open substack of $\bbR\Sec_{\Pic_d}(\ffL^{\oplus r+1}_d/\ffC_{\Pic_d})$.
\end{theorem}

From this discussion, we can write a point in $\RGw$ as $(C,L,s_0,\dots ,s_r)$ where $C$ is a genus $g$, $n$-marked prestable curve (we suppress the notation for the marked points), $L$ is a degree $d$ line bundle on $C$ and $s_0,\dots, s_r$ are sections. In this notation, the stability conditions of stable maps translate to the following.
\begin{definition}\label{conditions_st_maps_as_sections}[Stability conditions of stable maps as sections]
\begin{enumerate}
\item The bundle $\omega_C^{\log}\otimes L^{\otimes 3}$ is ample, which is a condition on the pair $(C,L)$ already present in $\Pic$,
\item The linear system $(L, s_0,\dots, s_r)$ has no base points.\label{condition:no_base_points_st_maps_as_sections}
\end{enumerate}
\end{definition}
\begin{corollary}\label{thm:potMgn}
There is a forgetful morphism $\Gw\to\Pic_d$ sending 
\[(C, L, s_0,\dots, s_r)\mapsto (C, L)\text{.}\]
The morphism is quasi-smooth with dual perfect obstruction theory
\[
\bbT_{\Gw/\Pic_d}\to \bbE_{\Gw/\Pic_d}=\rder^{\bullet}\pi_{\Gw *}\ccL_{\Gw}^{\oplus r+1}
\]
where $\ccL_{\Gw}$ is the locally-free sheaf on $\ffC_{\Gw}$ obtained from the map into $\Pic_d$. This perfect obstruction theory is compatible with the usual perfect obstruction theory of stable maps in the sense of \cite{Cristina-virtual-pullbacks-2012}.
\end{corollary}
\begin{proof}
This follows from the discussion above and \cref{cor:pot}.
\end{proof}

%%%%%%%%%%%%%%%%%%%%%%%%
\subsection{Quasi-maps to $\bbP^r$ as sections}\label{ss:qm}
%%%%%%%%%%%%%%%%%%%%%
We have another way of understanding maps of curves to $\bbP^r$, by relaxing the concept of map and allowing a linear system $(L, s_0, \dots , s_r)$ on a curve $C$ to develop some base points. Consider $\mathfrak{Pic}_d=\mathfrak{Pic}_{g,n,d}$ the usual stack parametrizing genus $g$, $n$-marked pre-stable curves with a degree $d$ bundle without stability conditions. Let $\ffC_{\mathfrak{Pic}_d}, \ffL_d$ denote the universal curve and universal bundle respectively.
\begin{definition}[Stable quasi-maps{\cite[Definition 3.1.1]{Fontanine-QMap2010}}]\label{def: quasi maps}
\[
\Qm\subset\Sec_{\mathfrak{Pic}_d}(\ffL_d^{\oplus r+1}/\ffC_{\mathfrak{Pic}_d})
\]
is the open substack defined by imposing following conditions on each geometric fiber

\begin{enumerate}
\item (non-degeneracy )The linear system $(L, s_0,\dots , s_r)$ has finitely many base points away from the nodes and the markings of $C$.
\item (stability) The line bundle $\omega_C^{\log}\otimes L^{\otimes \epsilon}>0$ for any $\epsilon\in\bbQ_{>0}$.
\end{enumerate}
\end{definition}
The derived enhancement $\Sec_{\mathfrak{Pic}_d}(\ffL_d^{\oplus r+1}/\ffC_{\mathfrak{Pic}_d})\xhookrightarrow{}\bbR \Sec_{\mathfrak{Pic}_d}(\ffL_d^{\oplus r+1}/\ffC_{\mathfrak{Pic}_d})$ gives a derived enhancement $\Qm\xhookrightarrow{j}\RQm$.
The usual perfect obstruction for the moduli of quasi-maps (eg. \cite{Fontanine-QMap2010}) comes from this derived extension.
Indeed, the computation in \cref{thm:tanS} shows that
\[
\bbT_{\Qm/\Picc}\to j^*\bbT_{\RQm/\Picc}=\rder\pi_*f^*\ccO_{\bbP^r}(1)\text{,}
\]
where as usual $\pi$ and $f$ are the universal projection and evaluation respectively from $\ffC_{\RQm}$.
We will see in the next section a slightly different construction of $\Qm$ that yields an equivalent derived enhancement.

%%%%%%%%%%%%%%%%%%%%%%%%%%%%%%%%%%%%%%%%%%%%%%%%%%%%%
%%%%%%%%%%%%%%%%%%%%%%%%%%%%%%%%%%%%%%%%%%%%%%%%%%%%%
\section{Stable maps and quasi-maps to $\bbP^r$}\label{sec:GWtoQM}
%%%%%%%%%%%%%%%%%%%%%%%%%%%%%%%%%%%%%%%%%%%%%%%%%%%%%
%%%%%%%%%%%%%%%%%%%%%%%%%%%%%%%%%%%%%%%%%%%%%%%%%%%%%

In this section, we construct a morphism between the derived enhancement of the moduli space of stable maps to $\bbP^r$ and of quasi maps that is 
\[
\oc:\RGw \to \RQm\text{.}
\]
We prove that 
\[
\oc_* \ccO_{\RGw} = \ccO_{\RQm} \mbox{ in } \ccD^b(Coh(\RQm)\text{.}
\]

%%%%%%%%%%%%%%%%%%%%%%%%%%%%%%%%%%%%%%%%%%%%%%%%%%%%%
\subsection{Revised notation}\label{notation}
%%%%%%%%%%%%%%%%%%%%%%%%%%%%%%%%%%%%%%%%%%%%%%%%%%%%%

From here, we will adopt a slightly different notation from that of the preceding sections in the interest of clarity.
Let $\ffM\coloneqq \ffM_{g,n}$ denote the moduli space of genus $g$ pre-stable curves with $n$ marked points and let $\upi:\uC\to\ffM$ denote its universal curve. Let $\Pic\coloneqq {{\Pic_{g,n,d} }}$ denote the moduli space defined in Example \ref{ex:pic} and let $\upi:\ffC\to\Pic$ denote its universal curve. Recall, that $\Pic$ parametrizes pairs $(C,L)$, with $C$ a prestable curve in $\ffM$ and $L$ is a line bundle of fixed degree $d$ over $C$ subject to the stability conditions in \ref{ex:pic}. Then we define $\ffC$ by the following cartesian diagram

\[
\begin{tikzcd}
\ffC \arrow[d, "\pi"] \arrow[r] \arrow[dr, phantom,"\ulcorner", very near start] 
& \uC \arrow[d, "\upi"] \\
\Pic \arrow[r] & \ffM\text{.}
\end{tikzcd}
\]
 
Notice that $\mathfrak{C} \to  \ffM $ is flat so that the stack fiber product is also the homotopical fiber product. 
Let $\ffL$ over $\ffC$ denote the universal bundle so we have
\begin{equation}\label{eq:L,C,Pic,M}
\begin{tikzcd}
\ffL \arrow[r] &  \ffC \ar[r,"\pi"] & \Pic \ar[r] & \ffM\text{.}
\end{tikzcd}
\end{equation}

\begin{definition}\label{defi:rational,tail}
 Let $C$ be a point in $\ffM$.  A \textit{rational tail} $\Gamma$ in $C$ is a maximal tree of rational components without marked points and such that  $\Gamma\cap \overline{C\setminus \Gamma}$ is a point.
\end{definition}
Let $\widehat{\ffM}$ denote the moduli space of pre-stable curves of genus $g$ with $n$ marked points without rational tails. Let $\upi:\widehat{\uC}\to\widehat{\ffM}$ denote the universal curve. In \cite[p.12]{Fontanine-QMap2010}, the authors prove that $\widehat{\ffM}$ is an open substack of finite type in $\ffM$, with universal curve isomorphic to the restriction of $\ffC$ to $\widehat{\ffM}$.

Let $\widehat{\mathfrak{Pic}}$ denote the Cartesian product
\[
\begin{tikzcd}
\widehat{\mathfrak{Pic}}\arrow[r]\arrow[d]\arrow[dr, phantom, "\ulcorner",very near start] & \mathfrak{Pic}\arrow[d]\\
\widehat{\ffM}\arrow[r] & \ffM\text{.}
\end{tikzcd}
\]
A closed point in $\wh{\mathfrak{Pic}}$ is a pair $(\wh{C},\wh{L})$ of a marked pre-stable curve with no rational tails and a line bundle.
\begin{definition}\label{def:wPic}
As in the case of $\Pic$ let $\wh\Pic$ denote the substack of $\widehat{\mathfrak{Pic}}$
with the additional stability conditions:
\begin{enumerate}
\item for any $\epsilon\in\bbQ_{>0}
$, we have \[\omega_{\wh C}^{\log}\otimes \wh L^{\otimes \epsilon}>0\text{,}\]  Here $\omega_{\wh{C}}^{\log}$ denotes the dualizing sheaf of the curve twisted by the sum of the $n$ marked points.
\item on all components $\wh C_i$ of $\wh C$, we have \[ \deg(\wh L)|_{\wh C_i}\geq 0\text{.}\] 
\end{enumerate}
\end{definition}
We have that $\wh{\Pic}$ is an open substack of $\wh{\mathfrak{Pic}}$.

We define $\wh{\ffC}$ as the following fiber product 
\[
\begin{tikzcd}
\wh{\ffC}\ar[r]\ar[d, "\wh\pi"] \ar[dr,phantom,"\ulcorner"
,very near start]& \wh{\uC}\ar[d, "\wh\upi"]\\
\widehat{\Pic}\ar[r]& \wh\ffM\text{.}
\end{tikzcd}
\]
We have a universal line bundle, denoted by $\wh{\ffL}$ over $\wh{\ffC}$.
As in \eqref{eq:L,C,Pic,M}, we have the following morphisms.
\begin{equation}\label{eq:hat,L,C,Pic,M}
\begin{tikzcd}
\wh{\ffL} \arrow[r] &  \wh{\ffC} \ar[r,"\wh{\pi}"] & \wh{\Pic} \ar[r] & \wh{\ffM}\text{.}
\end{tikzcd}
\end{equation}

%%%%%%%%%%%%%%%%%%%%%%%%%%%%%%%%%%%%%%%%%%%%%%%%%%%%%
\subsection{Quasi-maps are defined over $\wPic$}\label{ss:qm over pic hat}
%%%%%%%%%%%%%%%%%%%%%%%%%%%%%%%%%%%%%%%%%%%%%%%%%%%%%
In \cref{ss:qm} we defined
\begin{equation}\label{def:RQm_over_Pic}
\RQm\subset\bbR\Sec_{\mathfrak{Pic}}(\ffL^{\oplus r+1}/\ffC_{\mathfrak{Pic}})\text{.}
\end{equation}
Given the definitions of this sections, we have a new substack $\wPic\subset \mathfrak{Pic}$.
The stability conditions of quasi-maps imply that the source curve cannot have rational tails. So the morphism $\RQm\to\mathfrak{Pic}$ factors through $\wPicc$. Moreover, the stability conditions of $\RQm$ imply those of \cref{def:wPic}, so we obtain a morphism $\RQm\to\wPic\subset\mathfrak{Pic}$. 

By \cref{prop:restriction of sec}, we have that
\[
\bbR\Sec_{\wPic}(\wh{\ffL}^{\oplus r+1}/\wh{\ffC}) = \wPic\times^h_{\mathfrak{Pic}}\bbR\Sec_{\mathfrak{Pic}}(\ffL^{\oplus r+1}/\ffC_{\mathfrak{Pic}})\text{.}
\]
So the open embedding in \eqref{def:RQm_over_Pic} factors through an open embedding $\RQm\subset \bbR\Sec_{\wPic}(\wh{\ffL}^{\oplus r+1}/\wh{\ffC})$. 
We state the implications of this below.
\begin{proposition}
The moduli space of quasi-maps $\Qm$ has a forgetful morphism to the stack $\wPic$ of pre-stable curves with no rational tails with a stable line bundle (\cref{def:wPic}). Over $\wPic$ it admits an open embedding into the derived stack of sections $\bbR\Sec_{\wPic}(\wh{\ffL}^{\oplus r+1}/\wh{\ffC})$. This endows $\Qm$ with a derived enhancement $\RQm$ which is compatible with the derived enhancement from \cref{ss:qm}.

In particular, this derived enhancement recovers the canonical perfect obstruction theory of the moduli space of quasi-maps.
\end{proposition}
%%%%%%%%%%%%%%%%%%%%%%%%%%%%%%%%%%%%%%%%%%%%%%%%%%%%%
\subsection{Derived morphism between stable maps and quasi-maps}\label{map to qmap}
%%%%%%%%%%%%%%%%%%%%%%%%%%%%%%%%%%%%%%%%%%%%%%%%%%%%%
In this subsection, we want to construct the following.
\begin{enumerate}
    \item A commutative diagram 
    \begin{equation}\label{diag}
\begin{tikzcd}
\ffL \ar[d] \arrow[r]& \wh{\ffL} \ar[d]\\
\ffC \ar[r,"k"] \ar[d,"\pi"] & \widehat{\ffC} \ar[d,"\wh{\pi}"]\\
\Pic\ar[d] \ar[r,"c"] & \widehat{\Pic} \ar[d] \\ 
\ffM\ar[r,"\uc"] & \widehat{\ffM}
\end{tikzcd}
    \end{equation}
    which relates \eqref{eq:L,C,Pic,M} and \eqref{eq:hat,L,C,Pic,M}.
\item A morphism (see Proposition \ref{prop:div,rational,tails} below)
\[\oc :\bbR \Sec_{{\Pic }} (\ffL^{\oplus r+1}/\ffC) \to \bbR \Sec_{{\widehat{\Pic} }} (\widehat{\ffL}^{\oplus r+1}/\widehat{\ffC})
\] that restricts to a morphism $\bbR\overline{\ccM} (\bbP^r,d) \to \bbR\overline{\ccQ} (\bbP^r,d).$
\end{enumerate}

%%%%%%%%%%%%%%%%%%%%%%%%%%%%%%%%%%%%%%%%%%
\subsubsection{Construction of contraction morphism $\uc: \ffM\to\wh{\ffM}$}\label{subsubsection:contraction:M,hatM}
%%%%%%%%%%%%%%%%%%%%%%%%%%%%%%%%%%%%%%%%%%

Recall that in \cite[proof of Thm 7.1]{Popo-roth-stable-maps-Quot-2003} or \cite[Prop 2.3]{Cristina-stable-maps-quotients-2014}) one can construct a non separated\footnote{This morphism is not separated because, the trivial family $\bbP^1\times \bbA^1\setminus \{0\} \to \bbA^1\setminus \{0\}$ can be completed at $\{0\}$ by $\bbP^1$ or the blowup of the trivial family in any number of points in the special fibre.} morphism 
\[
\uc:\ffM\to\wh\ffM
\]
which contracts the rational tails. For $S\to\ffM$,

\begin{align*}
\uc: \ffM&\to\wh{\ffM}\\
(C_S,S)&\mapsto (\wh{C}_S,S)
\end{align*}
where $\wh{C}_S$ is the family $C_S$ with rational tails contracted in each fiber.
Recall that $\underline{\ffC}$ (resp. $\underline{\wh{\ffC}}$) is the universal curve of $\ffM$ (resp. $\wh{\ffM}$). Moreover, we have a commutative diagram which is not Cartesian
\begin{equation}\label{diag:M,hatM,curve,uni}
\begin{tikzcd}
\underline{\ffC}\ar[d,"\upi"] \ar[r,"\uk"] & \underline{\widehat{\ffC}} \ar[d,"\widehat{\upi}"]\\ \ffM\ar[r,"\uc"] & \widehat{\ffM}\text{.}
\end{tikzcd}
\end{equation}
Notice that $\uc$ is a birational morphism. 
%%%%%%%%%%%%%%%%%%%%%%%%%%%%%%%%%%%%%%%%%%
\subsubsection{Construction of contraction of tails morphism $c:\Pic \to \wh{\Pic}$} \label{subsubsection:Pic,to,whPic}
%%%%%%%%%%%%%%%%%%%%%%%%%%%%%%%%%%%%%%%%%%
Recall that we have \eqref{eq:L,C,Pic,M},
\[\begin{tikzcd}
\ffL \arrow[r] &  \ffC \ar[r,"\pi"] & \Pic \ar[r] & \ffM\text{.}
\end{tikzcd}
\]

\begin{definition}\label{defi:divisor,rational,tails}
Let denote $\ffD$ by the divisor in $\ffC$ consisting of rational tails. By considering the restriction of the universal bundle $\ffL$ on $\ffD$ we can split the divisor into 
\[
\ffD=\bigsqcup_{i=1}^d \ffD_i
\]
such that $\ffL|_{\ffD_i}$ has degree $\delta_i$.  
We write $\delta\ffD$ for $\sum_i\delta_i\ffD$.
\end{definition}

We first define $c:\Pic \to \wh{\Pic}$ at the level of points (see \cite[\S 2.2]{Cristina-stable-maps-quotients-2014}). Let $(C,L)\in \Pic$. Let $T_i$ be the rational tails of $C$ and let $\delta_i$ denote the total degree of $L$ on $T_i$. Notice that $\sum_i \delta_i=\deg(L|_{\sqcup_i T_i})$. Let $\widehat{C}$ be the closure of $C\backslash \bigcup_i T_i$. Let $Q_i$ denote the point $T_i\cap \wh C$. We define
\[\wh L:= L|_{\wh C}(\sum_i\delta_i Q_i)\text{.}
\]
In families we proceed similarly: let $S\to \Pic$ with a family of curves $C_S\to S$ and a line bundle $\ccL_S$.
We define $\wh{C}_S:=\underline{c}(C_S)$ contracting rational tails.
We put
\[
\wh{\ccL}_S:=\ccL|_{\wh{C_S}}\left(\delta \ffD\right)\text{,}
\]
and we obtain a morphism 
\begin{align*}
    c: \Pic &\to \wh{\mathfrak{Pic}} \\
     (C_S,\ccL_S) & \mapsto (\wh{C}_S,\wh{\ccL}_{S})\text{.}
\end{align*}
To show it factors as the required morphism
\[
c:\Pic\to\wPic
\]
we need to check that the following are true:
\begin{enumerate}
    \item If $L$ has non-negative degree on each component of $C$, then $\wh{L}$ has non-negative degree on each component of $\wh{C}$. 
    \item If $\omega_C^{\log}\otimes L^{\otimes 3}>0$, then $\omega_{\wh{C}}^{\log}\otimes \wh{L}^{\otimes \epsilon}>0$ for all $\epsilon\in\bbQ_{>0}$.
\end{enumerate}
The first statement is clear. 
The only case where the first condition does not immediately imply the second is that of a genus 0 component $C_i$ with less than two marked points. The first condition then requires that the degree of $L|_{C_i}$ is at least 1. Note that any component of $\wh{C}$ has at least one marked point, so the degree of $\omega_{\wh{C}}^{\log}$ is greater or equal than $-1$.  This shows that the degree of $\wh{L}|_{\uc({C_i})}$ is at least 1 and thus the claim.

We thus get the following commutative diagram
\begin{equation}
\begin{tikzcd}
\Pic\ar[d] \ar[r,"c"] & \widehat{\Pic} \ar[d] \\ 
\ffM\ar[r,"\uc"] & \widehat{\ffM}\text{.}
\end{tikzcd}
    \end{equation}

%%%%%%%%%%%%%%%%%%%%%%%%%%%%%%%%%%%%%%%%%%
\subsubsection{Construction of the morphism $k:\ffC \to \wh{\ffC}$} \label{subsubsection:C,to,whC}
%%%%%%%%%%%%%%%%%%%%%%%%%%%%%%%%%%%%%%%%%%
 As $\ffC:=\underline{\ffC}\times_\ffM \Pic$ (resp. $\wh{\ffC}:=\wh{\underline{\ffC}}\times_{\wh{\ffM}} \wh{\Pic}$) and the Cartesian diagram \eqref{diag:M,hatM,curve,uni}, writing all the diagrams, we get the morphism $k:\ffC \to \wh{\ffC}$ such that the following diagram is commutative 
\begin{equation}
\begin{tikzcd}
\ffC \ar[r,"k"] \ar[d,"\pi"] & \widehat{\ffC} \ar[d,"\wh{\pi}"]\\
\Pic\ar[d] \ar[r,"c"] & \widehat{\Pic} \ar[d] \\ 
\ffM\ar[r,"\uc"] & \widehat{\ffM}\text{.}
\end{tikzcd}
    \end{equation}

%%%%%%%%%%%%%%%%%%%%%%%%%%%%%%%%%%%%%%%%%%
\subsubsection{Construction of the morphism $\ffL \to \wh{\ffL}$} \label{subsubsection:L,to,whL}
%%%%%%%%%%%%%%%%%%%%%%%%%%%%%%%%%%%%%%%%%%
We can decompose the morphism $k$ as $\ell\circ\kappa$ as in the following diagram:
\begin{equation}\label{diag:big}
\begin{tikzcd}
\ffL \ar[d] & \ell^*\wh{\ffL}\ar[d] \ar[r] \ar[rd,phantom,"\ulcorner", very near start]&\wh{\ffL} \ar[d]\\
\ffC \ar[r,"\kappa"] \ar[rd,"\pi"]& c^*\wh{\ffC} \ar[d,"c^*\wh{\pi}"] \ar[r,"\ell"]  \ar[rd,phantom,"\ulcorner", very near start]& \widehat{\ffC} \ar[d,"\wh{\pi}"]\\
&\Pic\ar[d] \ar[r,"c"] & \widehat{\Pic} \ar[d] \\ 
&\ffM\ar[r,"\uc"] & \widehat{\ffM}\text{.}
\end{tikzcd}
\end{equation}
Notice that $\wh{\pi}, \pi$ and $c^*\wh{\pi}$ are projective, so  $\kappa$ is projective. As $\kappa$ is birational and $\ccL(\delta \ffD)$ is trivial on rational tails, we have that 
\[\rder^1 \kappa_*\ccL(\delta \ffD)=0\text{.}\]
Then $\rder^0\kappa_*\ccL(\delta \ffD)$ is a line bundle on $c^*\wh{\ffC}$.
(See \cite[Lemma 7.1 and p.652-654]{Popo-roth-stable-maps-Quot-2003}. 

\begin{claim}\label{claim:pullback,line}
We have that 
\[
\rder^0{\kappa}_*\ccL(\delta \ffD) = \ell^* \wh{\ccL} \text{.}
\]
\end{claim}
\begin{proof}[Proof of the claim \ref{claim:pullback,line}]
As $\kappa$ is birational, the two sheaves are isomorphic away from the tails. On the tails, both are trivial.
On a smooth atlas of $c^*\wh{\ffC}$, they are isomorphic away from the locus where the tails are attached to the curve which is of codimension  $2$.
We deduce the statement.
\end{proof}

At the level of sheaves we have :
\[
\ccL \to \ccL (\delta \ffD) \mbox{ and by adjunction } \kappa^*\kappa_* \ccL (\delta \ffD) \to \ccL (\delta \ffD)\text{.}
\]

Notice that $\kappa^*\kappa_* \ccL (\delta \ffD)=\ccL(\delta \ffD)$ because both are isomorphic outside tails and trivial on tails.
Finally, we get a morphism from 
\[
\ccL \to \ccL(\delta \ffD)=\kappa^*\kappa_*\ccL(\delta \ffD)=\kappa^*\ell^*\wh{\ccL}\text{,}
\]
which leads to a morphism $\ffL \to \kappa_*\ffL(\delta \ffD)=\ell^*\wh{\ffL}$ that fills the diagram \eqref{diag:big}.

%%%%%%%%%%%%%%%%%%%%%%%%%%%%%%%%%%%%%%%%%%%%%
\subsubsection{Construction of the morphism $\oc :\bbR \Sec_{{\Pic }} (\ffL^{\oplus r+1}/\ffC) \to \bbR \Sec_{{\widehat{\Pic} }} (\widehat{\ffL}^{\oplus r+1}/\widehat{\ffC})$}\label{subsub:morphism,between,secL,secLhat}
%%%%%%%%%%%%%%%%%%%%%%%%%%%%%%%%%%%%%%%%%%%%

\begin{theorem}\label{prop:div,rational,tails}
We have a morphism
\[\oc :\bbR \Sec_{{\Pic }} (\ffL^{\oplus r+1}/\ffC) \to \bbR \Sec_{{\widehat{\Pic} }} (\widehat{\ffL}^{\oplus r+1}/\widehat{\ffC})\text{.}
\]
Moreover, the restriction of $\oc$ to $\RGw$ factors through \[\RQm\subset \bbR \Sec_{{\widehat{\Pic} }} (\widehat{\ffL}^{\oplus r+1}/\widehat{\ffC}),\] giving a morphism, denoted by the same name,
\[
\oc:\RGw\to\RQm \text{.}
\]
\end{theorem}
\begin{proof}
Multiplication by the canonical section gives a morphism $a:\ccL \to \ccL(\delta\ffD)$. We have the divisor exact sequence  
\begin{equation}
   0\to \ccL \to \ccL(\delta\ffD)\to \ccL(\delta\ffD)|_{\delta\ffD}\to 0
\end{equation}
over $\Pic$.   

The morphism $a$ of sheaves induces a morphism $\ffL\to \ffL(\delta\ffD)$ of total spaces, which induces
\begin{equation}\label{morph-mult}
   \bbR \Sec_{{\Pic }} (\ffL/\ffC)\to \bbR \Sec_{{\Pic }}(\ffL(\delta\ffD)/\ffC)\text{.}
\end{equation}
Now recall the locally-free sheaves
\[
\begin{tikzcd}
\ccL(\delta\ffD)\ar[r]\ar[d] & \rder^0\kappa_*\ccL(\delta\ffD)=\ell^*\wh{\ccL}\ar[d]\\
\ffC\ar[r, "\kappa"]\ar[dr,"\pi"] & c^*\wh{\ffC}\ar[d,"c^*\wh{\pi}"]\\
& \Pic
\end{tikzcd}
\]
and let $\kappa_*\ffL(\delta\ffD)$ denote the total space of $\rder^0\kappa_*\ccL(\delta\ffD)$.
There is an equivalence
\begin{align}
\bbR \Sec_{{\Pic }}(\ffL(\delta\ffD)/\ffC)&=\Spec_{\Pic}\Sym(\rder\pi_*\ccL(\delta\ffD))^{\vee} \nonumber\\
&\simeq\Spec_{\Pic}\Sym(\rder(c^*\wh{\pi})_*\kappa_*\ccL(\delta\ffD))^{\vee}\nonumber\\
&=\bbR\Sec_{\Pic}(\kappa_*\ffL(\delta\ffD)/c^*\wh{\ffC})\text{.} \label{eq:1}
\end{align}
This equivalence is simply a restatement of the fact that the sections of the push-forward of a sheaf on an open are sections of the original sheaf on the preimage.
Then by claim \ref{claim:pullback,line} and \eqref{eq:1}, we have
\begin{equation}
\bbR \Sec_{{\Pic }}(\ffL(\delta\ffD)/\ffC)\simeq \bbR\Sec_{\Pic}(\kappa_*\ffL(\delta\ffD)/c^*\wh{\ffC}) \simeq \bbR\Sec_{\Pic}(\ell^*\wh{\ffL}/c^*\wh{\ffC})\text{.}
\end{equation}
Now we just have to construct a morphism 
\[
\bbR\Sec_{\Pic}(\ell^*\wh{\ffL}/c^*\wh{\ffC})\to \bbR \Sec_{{\wPic }}(\wh{\ffL}/\wh{\ffC})\text{.}
\]

Let us consider the cartesian diagram
\[
\begin{tikzcd}
c^*\wh\ffC\ar[r, "\ell"] \ar[d,"\rho=c^*\wh{\pi}"']&\wh\ffC\ar[d,"\wh\pi"]\\
\Pic\ar[r,"c"]&\wh\Pic\text{.}
\end{tikzcd}
\]
By cohomology and base change, we get an isomorphism
\[ \rder\rho_* \ell^*\wh\ccL\to c^*\rder\wh\pi_*\wh\ccL\text{,}
\]
 that is we deduce that at the level of spaces, we have  
 \[\bbR \Sec_{{\Pic }} \left(\ffL\left(\delta\ffD\right)/\ffC\right)  \simeq c^*\bbR \Sec_{{\widehat{\Pic} }} (\widehat{\ffL}/\widehat{\ffC})=\bbR\Sec_{\Pic}(\ell^*\wh{\ffL}/c^*\wh{\ffC})\text{.}\]

By composing, we deduce a morphism
 \begin{equation}\label{morph-r-sec}
 \bbR \Sec_{{\Pic }} (\ffL(\delta\ffD)/\ffC) \simeq c^*\bbR \Sec_{{\widehat{\Pic} }} (\widehat{\ffL}/\widehat{\ffC}) \to \bbR \Sec_{{\widehat{\Pic} }} (\widehat{\ffL}/\widehat{\ffC})\text{.} 
 \end{equation}
 
Composing \eqref{morph-mult} with \eqref{morph-r-sec} we get a morphism 
\[
\bbR \Sec_{{\Pic }} (\ffL/\ffC) \to \bbR \Sec_{{\Pic }} (\ffL(\delta\ffD)/\ffC)  {\to} \bbR \Sec_{{\widehat{\Pic} }} (\widehat{\ffL}/\widehat{\ffC}) \text{.}
 \]
 
By applying the same argument to $\ffL^{\oplus r+1}$, we deduce the desired morphism $\oc$  
\begin{equation}\label{oc_on_sec}
\xymatrix{
\bbR \Sec_{{\Pic }} (\ffL^{\oplus r+1}/\ffC) \ar[r]^\oc\ar[d] & \bbR \Sec_{{\widehat{\Pic} }} (\widehat{\ffL}^{\oplus r+1}/\widehat{\ffC}) \ar[d] \\
{{\Pic }} \ar[r]^{c} & {\widehat{\Pic} }\text{.}
}
\end{equation}
Now we are left to check that the restriction of $\oc$ to $\RGw$ takes image in $\RQm$. We can check this on points, let $(C, L, s_0, \dots, s_r)\in\RGw$. We need to see that stability conditions of stable maps on $(C, L, s_0, \dots, s_r)$ imply those of quasi-maps on $\oc(C, L, s_0, \dots, s_r)$. The conditions about the ampleness of the bundles, are already checked at the level of $c:\Pic\to\wPic$. We only need to show that if $(L, s_0, \dots, s_r)$ has no base points, then $(\wh{L}, \wh{s}_0, \dots, \wh{s}_r)$ has finitely many base points away from markings and nodes. Let $Q_i$ be the attaching nodes of the rational tail $T_i$ on $C$. The only base points that are acquired by applying $\oc$ are on the images of the $Q_i$s in $\wh{C}$, but these are smooth and unmarked points of $\wh{C}$.

Thus we have a well-defined map given by the restriction of \ref{oc_on_sec}, which we will still denote by the same name:
\[
\oc:\RGw\to\RQm\text{.}
\]

\end{proof}

%%%%%%%%%%%%%%%%%%%%%%%%%%%%%%%%%%%%%%%%%%%%%%%%%%%%%
\section{Local embeddings} \label{subsec:local,embeddings}
%%%%%%%%%%%%%%%%%%%%%%%%%%%%%%%%%%%%%%%%%%%%%%%%%%%%%
The idea of this section is to control the map $\oc$ locally. In this section, we keep $\oc$ for the restriction morphism.

For any point $\xi\in\RGw$ we construct 
\begin{enumerate}
\item an open $\bbR V\subset \RGw$ which is a neighbourhood of $\xi$, 
\item an open $\bbR\widehat{V}\subset\RQm$ which is a neighbourhood of $\wh{\xi}:=\oc(\xi)$, such that $\oc:\bbR V \to \bbR\wh{V}$,
\item smooth Artin stacks  $W$ and $\wh{W}$ with a morphism $q:W\to \wh{W}$ which is proper and birational,
\item a vector bundle $\widehat{F}$ on $\widehat{W}$ together with a section $\theta$ such that 
\begin{itemize}
    \item the homotopical zero locus of $\theta$ is $\bbR \wh{V}$,
    \item the homotopical zero locus of $q^*\theta$ is $\bbR V$.
\end{itemize}

\end{enumerate}

Let us sum up the situation in the following diagram, where each square is Cartesian.
\begin{equation}\label{diag:local,total}
\begin{tikzcd}
\bbR V \ar[rrd,phantom,"\ulcorner_h", very near start] \ar[rdd,phantom,"\ulcorner_h", very near start]\ar[ddd] \ar[rd,"\oc"'] \ar[rrr]& & & W \ar[ld,"q"']\ar[ddd,"0"] \ar[ddl,phantom,"\urcorner_h",  near start]\\
& \bbR \wh{V} \ar[rd,phantom,"\ulcorner_h", very near start]\ar[d] \ar[r]& \wh{W} \ar[d,"0"] & \\
& \wh{W} \ar[r,"\theta"] & F &\\
W \ar[rru,phantom,"\llcorner_h", near start]\ar[ru,"q"] \ar[rrr,"q^*\theta"] &  & & q^* F \ar[lu]
\end{tikzcd}
\end{equation}
Practically, we have that 
\[
\ccO_{\bbR \wh{V}} = \Kos(F,\theta) \mbox{ and } \ccO_{\bbR V} = \Kos(q^*F,q^*\theta)\text{.}
\]
Notice that the right and bottom squares are homotopically Cartesian by
\cite[{Lemma 08I6}]{stacks-project}.

First, we will construct a different collection of open sets: $\bbR U$, $\UU$ and $\bbR\wh U$, $\wh \UU$. We will have that $\UU$ (resp. $\wh{\UU}$) is smooth and $\bbR U$ (resp. $\bbR \wh U$) sits inside it as a derived vanishing locus. These are all modulis of sections with minor stability. Later, by imposing the full stability conditions for stable maps and quasi-maps respectively, we will obtain $\bbR V$, $W$ and $\bbR \wh V$, $\wh W$.

\subsection{Constructions}\label{sec:constructions}

Fix $\xi:=(C', L', s'_0,\dots , s'_r)\in\RGw$ a closed point and $\wh\xi:=\oc(\xi)=(\widehat{C}',\widehat{L}', \wh{s}'_0,\dots ,\wh{s}'_r)\in\RQm$.

For any point $(\wh C, \wh L, \wh s_0\dots , \wh s_r)\in\RQm$, let $BL(\wh{s})=\cap_{i=0}^r Z(\wh{s}_i)$ be the base locus of $(\wh{s}_0,\ldots , \wh{s}_r)$. Note that by the definition of $\RQm$, we have that $BL(\wh{s})\to \RQm$ is a finite morphism. 
By possibly restricting $\wPic$ we may assume that we have an ample section of $\widehat{\ffC}\to\wPic$. We call the corresponding divisor $\wh{A}$. After replacing $\widehat{A}$ with an appropriate multiple, we may assume that 
\begin{equation}\label{no r1hat}
\rder^1\widehat{\pi}_* \widehat{L}(\widehat{A})=0\text{.}
\end{equation}
We further restrict this neighbourhood to a neighborhood of $c(C,L)$ such that there exists a section of $\mathcal{O}(\wh{A})$ (i.e. $\widehat{A}$ is a divisor of degree $a$ on $\widehat{\ffC}\to\wPic$) such that
\begin{enumerate}
    \item $\wh{A}$ does not contain 1-dimensional fibers of $\widehat{\ffC}\to\wPic$ and 
    \item for all points $(\wh{C},\wh{L})$ in a neighborhood of $c(C',L')$, $\wh{A}$ is disjoint from $BL(\wh{s})$ and special points of $\wh C$.
\end{enumerate}  We call this neighborhood $\widehat{\ffU}\subset \wPic$.

 Let $\ffU$ be $c^{-1}(\widehat{\ffU})$. We put $A\coloneqq k^*\widehat{A}$. By the fact that $L$ has positive degree on rational tails (see \cref{ex:pic} for the definition of $\Pic$), we have
 \begin{equation}\label{no r1}
\rder^1\pi_* L(A)=0
\end{equation}
on $\ffU$\text{.} 
We define
\[
\begin{tikzcd}
 \bbR U\coloneqq \bbR \Sec_\ffU (\ffL^{\oplus r+1}_\ffU/\ffC_{\ffU})
& \UU\coloneqq\bbR\Sec_\ffU(\ffL_\ffU(A)^{\oplus r+1}/\ffC_{\ffU})\\
 \bbR\widehat{U}\coloneqq\bbR\Sec_{\widehat{\ffU}}(\widehat{\ffL}^{\oplus r+1}_{\widehat{\ffU}}/\widehat{\ffC}_{\widehat{\ffU}})
& \widehat{\UU}\coloneqq\bbR\Sec_{\widehat{\ffU}}(\widehat{\ffL}_{\widehat{\ffU}}^{\oplus r+1}(\widehat{A})/\ffC_{\widehat{\ffU}})\text{.}\\
\end{tikzcd}
\]
Note that $\bbR U$ and $\bbR\wh{U}$ are open in $\bbR\Sec_{\Pic}(\ffL^{\oplus r+1}/\ffC)$ and $\bbR\Sec_{\wPic}(\widehat{\ffL}^{\oplus r+1}/\widehat{\ffC})$ respectively by \cref{prop:restriction of sec}. Moreover, $\xi\in \bbR U$ and $c(\xi)\in\bbR\wh{U}$.
By (\ref{no r1hat}) and (\ref{no r1}) we see that $\UU$ and $\widehat{\UU}$ are smooth and have no derived structure.
The morphism of sheaves $\ccL_{\ffU}\to\ccL_{\ffU}(A)$ gives a morphism $\bbR U\to \UU$. Similarly, we have $\bbR\wh{U}\to\wh{\UU}$. As in Theorem \ref{prop:div,rational,tails} we have a morphism $\tilde{q}: \UU\to \wh{\UU}$

\begin{construction}[Labelling of base points]\label{rk: decomposition of Zw}

Let $\wh\zeta=(\wh{C}', \wh{L}',\wh{w}'_0 ,\dots,\wh {w}'_r)$ be the image of $\wh\xi$ in $\wh\UU$. Let $BL(\wh {s}')$ be the base locus of $(\wh{s}'_0,\dots ,\wh {s}'_r)$, and $BL(\wh{w}')$ be the base locus of $(\wh {w}'_0,\ldots ,\wh {w}'_r)$. 

Notice also that by construction, we have that the base locus $BL(\wh{w}')=BL(\wh{s}')\sqcup \wh{A}$. This follows because $BL(\wh{s}')$ and $\wh{A}$ are disjoint by construction and for each $i$, $\wh{w}'_i$ is obtained by multiplying $\wh{s}'_i$ by the local defining equation of $\wh{A}$.  Then that we have a labelling
\[
BL(\wh{{w}}')=\underbrace{\{\wh{w}'_0=\dots= \wh{w}'_r=0\}\cap \wh{A}} _{BL(\wh{{w}'})_{\wh A}}\sqcup \underbrace{\{\wh s'_0= \dots =\wh s'_r\}}_{BL(\wh{{w}}')_{\wh L}}\text{.}
\]

In the following we construct an analytic (or \'etale) neighbourhood of $\wh\zeta$ such that the labelling above can be extended on this neighbourhood. First, we restrict $\wh\UU$ to the open where the base locus $BL(\wh{w})$ does not contain 1-dimensional fibers of $\ffC\to \Pic$. On this open set we have that $BL(\wh{w}')\to \wh\UU$ is proper and finite and over $\wh\zeta$ we have that $BL(\wh{{w}}')_{\wh{A}}$ and $BL(\wh{{w}}')_{\wh{L}}$ lie on different connected components. Then, there exists an analytic (or \'etale) neighbourhood $\wh{\UU'}$ such that the restriction of $BL(\wh{ w})$ to $\UU'$ can be written as a union of disconnected components $BL(\wh{{w}})_{\wh A}$ and $BL(\wh{{w}})_{\wh L}$ which contain $BL(\wh{{w}}')_{\wh A}$ and $BL(\wh{{w}}')_{\wh L}$ respectively.
This means that for a point $(\wh C, \wh L, \wh w_0,\dots ,\wh w_r)\in \wh{\UU'}$ the base points $BL(\wh{{w}})$ of $(\wh L, \wh w_0,\dots ,\wh w_r)$ are labelled by the connected components of the base locus
\[
BL(\wh {{w}})=BL(\wh {{w}})_{\wh A}\sqcup BL(\wh{w})_{\wh L}\text{.}
\]

\end{construction}
Now we define $W\subset\UU'$ and $\wh{W}\subset\wh{\UU'}$ by imposing some stability conditions. 
\begin{construction}[Construction of $\wh{W}$]\label{construction W hat}
Let \[(\wh{C}, \wh{L}, \wh{w}_0,\dots , \wh{w}_r)\in \wh{\UU'}\subset\Sec_{\wh{\ffU}}(\wh{\ffL}_{\wh{\ffU}}(\wh{A})/\wh{\ffC}_{\wh{\ffU}})\text{.}\]  This point is in $\wh{W}$ if
\begin{enumerate}[label=(\roman*)]
\item the base locus of $\wh{w}_0,\dots , \wh{w}_r$ is discrete and disjoint from all the special points of $\wh{C}$,
\item for any $\epsilon\in\bbQ_{>0}$,
\[
\omega_{\wh{C}}^{\log}\otimes \wh{L}(\wh{A})^{\otimes \epsilon}>0\text{.}
\]
\end{enumerate}
Note that the base locus is labelled in the sense of construction \ref{rk: decomposition of Zw} because we are in $\wh{\UU '}$.
\end{construction}
\begin{remark}\label{affine w hat} We have that $\wh W$ is a smooth as it is an open in $\mathbb{V}(\rder\pi_{\mathfrak{U}\ast}\mathfrak{L_U}(A))$ and $\mathbb{V}(\rder\pi_{\mathfrak{U}\ast}\mathfrak{L_U}(A))$ is smooth as $\rder^{1}\pi_{\mathfrak{U}\ast}\mathfrak{L_U}(A)=0$. Then, $\wh W$ is a Deligne--Mumford (DM) stack, as all points have finite automorphism; this is because the conditions in  Construction   \ref{construction W} are a mixed between stable maps and quasi-maps conditions which are both DM stacks. By possibly shrinking $\wh W$ and \cite[Theorem 1.1]{Alper-Hall-Rudh-Luna-slice-2020} (see also \cite{Abra-Vistoli-compactifying-stable-maps-2002}, \cite{Kresch-geo-DM-2009}), we can also assume that $\wh{W}$ is the quotient of a smooth affine scheme by a finite group.

Again by construction, $\wh W$ is the stack quotient $[Y/G]$ with $Y$ smooth and affine (i.e., $Y=\Spec A$) and $G$ a finite group. As this group $G$ is inside a linear group, then its coarse moduli space is $\Spec A^G$. We have that $\Spec A^G$ is an algebra with finite generators hence it is quasi-projective.
\end{remark}

\begin{construction}[Construction of $W$]\label{construction W}
Let $\UU'=\wh{\UU'}\times_{\wh{\UU}}\UU$. We have that $\UU'$ contains the point $\zeta$, the image of $\xi$ in $\UU$. By construction, we have a map $\mathfrak{q}:\UU'\to \wh{\UU'}$ and an induced map between universal curves $\mathfrak{k}$. If for any point $(C,L, w)\in \UU'$, we denote its image under $\mathfrak{q}$ by $(\wh C, \wh L, \wh w)$, then we have that $\mathfrak{k}$ maps $BL(w)$ to $BL(\wh w)$. Since the base locus in $\wh{\UU'}$ is labelled, we have that the base locus in ${\UU'}$ is labelled: 
\[BL(w)=BL(w)_{A}\sqcup  BL(w)_{L}.
\]
Let \[(C, L, w_0,\dots ,w_r)\in \UU'\subset\Sec_{\ffU}(\ffL_{\ffU}(A)/\ffC_{\ffU})\text{.}\]  This point is in $W$ if
\begin{enumerate}[label=(\roman*)]
\item the base locus $BL(w)$ of $w_0, \dots , w_r$ is discrete and disjoint from the special points of $C$, 

\item the subset $BL(w)_{L}$ of the base locus is empty and
\item the line bundle $\omega_{C}^{\log}\otimes L^{\otimes 3}$ is ample.
\end{enumerate}
Notice that by the definition of $\UU'$ we have that the base locus is labelled and thus condition (ii) makes sense.
\end{construction}

\begin{remark}
Notice that for any $(C, L, w_0,\dots ,w_r)\in W$, the choice of $\wh{A}$ and the stability condition imply that $A\mid_{C}$ does not intersect rational tails for any $C$. 

This was not the case for points in $\UU'$ without the stability condition in (i) in Construction \ref{construction W}.
\end{remark}

We define $\bbR V=\bbR U\times_{\UU}W$ and $\bbR\widehat{V}=\bbR\widehat{U}\times_{\wh\UU}\wh{W}$.

\begin{remark}
    The idea behind the construction is to define compatible atlases on $\RGw$ and $\RQm$, in the sense that we want open covers by $\bbR V$ and $\bbR \wh{V}$ respectively such that
    \[\begin{tikzcd}
        \bbR V\ar[d]\ar[r]\ar[dr, phantom, very near start, "\ulcorner^h"]&\RGw\ar[d,"c"]\\
        \bbR\wh{V}\ar[r]& \RQm \text{.}
    \end{tikzcd}
    \]
    In addition, we want $\bbR V$ and $\bbR\wh{V}$ to be derived vanishing loci of triples $(W,F,\theta)$ and $(\wh{W},\wh{F},\wh{\theta})$ where the first family is a pullback of the second. These are triples of a smooth algebraic (non-derived) stack, a vector bundle and a section.
    To achieve this, we start by covering $\RGw$ and $\RQm$ by sets open in $\SecM$ and $\SecQ$. These sets will be of the form $\bbR U=\bbR\Sec_{\ffU}(\ffL^{\oplus r+1}/\ffC_{\ffU})$, $\bbR \wh{U}=\bbR\Sec_{\wh{\ffU}}(\wL^{\oplus r+1}/\wC_{\wh{\ffU}})$. They are chosen so that it is possible to pick sufficiently ample divisors on $\wh{A}$ and $A$ on $\wC_{\wh{\ffU}}$ and $\ffC_{\ffU}$ which are away from the rational tails and base points and give $\bbR U$ and $\bbR\wh{U}$ smooth embeddings (see \cref{prop:UandUU} and \cref{lem:local} for more details). 
    
    We end up with a closed embedding $m_A: \bbR U\to\UU=\bbR\Sec_{\ffU}(\ffL(A)^{\oplus r+1}/\ffC_{\ffU})$ and a similar one for $\bbR\wh{U}$. Here, $m_A(\bbR U)$ is the space of $(r+1)$-tuples of sections of $\ffL(A)$ which are all divisible by the local equation of $A$.
    Now could define $\bbR\wh{V}=\bbR \wh{U}\cap\RQm$ and $\bbR V=\bbR U\cap\RGw$, but we have found that more care is needed at this stage. 
    
    The problem is finding an open $W\subset \UU$ such that $W\cap m_A(\bbR U)=\bbR V$, and similarly for the contracted ``hat-versions''. We cannot specify ``stability conditions'' on $\UU$ that will restrict to those of stable maps when restricted to the subvariety $\bbR U$, if we also want to do this construction in a parallel and compatible way on the hat-versions. Thus, we restrict to an analytic open set around $m_A(\bbR U)$, defined by $(r+1)$-tuples of sections of $\ffL(A)$ which are all divisible by small deformations of the local equations of $A$ (See Figure \ref{fig:rem,construction,W} for the intuition). Restricting to such an analytic open subset allows us to impose the condition $\UU$ that the $(r+1)$-tuple of sections of $\ffL(A)$ has no base points, apart from those very close to $A$. This is the reason we pass to the \'{e}tale topology in the construction. 

\begin{figure}[ht]\label{fig:rem,construction,W}
    \centering
\begin{tikzpicture}[scale=2]

%%%%Three plans
    \filldraw[blue!90!green,fill opacity =0.1]  (0.3,0) -- (2.7,0) -- (3.5,1) -- (1.1,1) -- (0.3,0);
%    \filldraw[blue!90!green,fill opacity =0.1]  (0,0) -- (3,0) -- (4,1) -- (1,1) -- (0,0);
    \draw[dashed] (0,0.7) -- (3,0.7) -- (4,1.7) -- (1,1.7) -- (0,0.7);
    \draw[dashed] (0,-0.7) -- (3,-0.7) -- (4,0.3) -- (1,0.3) -- (0,-0.7);
    \draw (2,0.45) node {\textbullet};   
    \draw[blue!90!green,fill opacity =1] (0.7,0.1) node {$\bbR U$};
    \draw[blue!80!green,fill opacity =1] (2,0.1) node {$\bbR V$};
    \draw[red,fill opacity =1] (2,1.2) node {$W$};
    \draw[black,fill opacity =1] (-0.3,0.2) node {$\UU$};

%%% ngbh
\filldraw[blue!80!green,fill opacity =0.3] (2,0.45) ellipse (0.5 and 0.2);
%\filldraw[fill=blue!40!white, draw=black] (2,0.45) ellipse (0.5 and 0.2);
\draw[dashed,red] (2,-0.25) ellipse (0.8 and 0.2);
\draw[dashed,red] (2,1.25) ellipse (0.8 and 0.2);
%\draw (2,-0.25) node  {\textbullet};
%\draw (1.5,0.45) node  {\textbullet};
%\draw (2,-0.25) node  {\textbullet};
%\draw (2,-0.25) node  {\textbullet};

%    \draw[rounded corners] (2.5,0.65) -- (2,0.2) -- (2.5,0.5) -- (2,0.7) -- (1.5,0.5);
\draw[dashed] (0,0.7) .. controls (0.4,0)  .. (0,-0.7);
\draw[dashed] (3,0.7) .. controls (2.6,0)  .. (3,-0.7);
\draw[dashed] (4,1.7) .. controls (3.35,1)  .. (4,0.3);
\draw[dashed] (1,1.7) .. controls (1.15,1)  .. (1,0.3);

\draw[red] (1.2,-0.25) .. controls (1.6,0.45)  .. (1.2,1.25);
\draw[red] (2.8,-0.25) .. controls (2.4,0.45)  .. (2.8,1.25);

%\draw (1.2,-0.25) -- (1.2,1.25);
%\draw (2.8,-0.25) -- (2.8,1.25);

 \end{tikzpicture}
    \caption{The ambient space is $\UU$ which is an open in the moduli of sections of $\ffL(A)$, $\bbR U$ is an open in the moduli of sections of $\ffL$, $\bbR V$ is $\bbR U\cap \RGw$.  
    We draw this picture for stable maps (without hat) but we should imagine the same for quasi-maps in a compatible way.
    }
    \label{fig:my_label}
\end{figure}
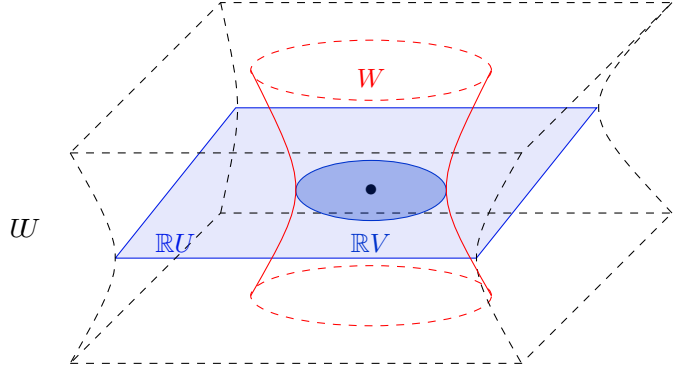

\end{remark}

\subsection{Properties}
\begin{proposition}\label{prop:UandUU}
There exists a vector bundle $E$ on $\UU$ and a section $\sigma$ such that $\bbR U$ is the derived zero locus of $\sigma$ that is $\ccO_{\bbR U}= \Kos(\ccE,\sigma)$ where $\ccE$ is the sheaf of sections of $E$.

Similarly, there exists a vector bundle $\wh{E}$ on $\wh{\UU}$ and a section ${\wh{\sigma}}$ such that $\ccO_{\bbR \wh{U}}=\Kos(\wh{\ccE},\wh{\sigma})$.
\end{proposition}

\begin{proof}
Recall from \eqref{no r1} and \eqref{ses u}  that we have $\rder^1\pi_{\ffU *}\ccL_\ffU (A)=\rder^1\wh{\pi}_{\ffU *}\wh{\ccL}_\ffU(\wh{A})=0$.
Multiplying by a local equation of $A$ and pushing forward gives a distinguished triangle of sheaves on $\ffU$.
\begin{equation}\label{ses u} \rder\pi_{\ffU *}\ccL_{\ffU}\to \rder\pi_{\ffU *}\ccL_{\ffU}(A)\xrightarrow{s} \rder\pi_{\ffU *}\ccL_{\ffU}(A)|_{A}\xrightarrow{+1}
\end{equation}
Observe that $\bbR U=\bbV( \rder\pi_{\ffU *}\ffL_{\ffU})$, and 
\[\UU=\bbV(\rder\pi_{\ffU *}\ffL_{\ffU}(A))=\bbV(\pi_{\ffU *}\ffL_{\ffU}(A))\text{.}
\]
We also have $ \rder^{1}\pi_{\ffU *}\ffL_{\ffU}(A)|_{A}= 0$, forced by the long exact sequence of (\ref{no r1}) and (\ref{ses u}). Then $\ffE\coloneqq \bbV(\pi_{\ffU *}\ffL_{\ffU}(A)|_{A})$ is a non-derived vector bundle on $\ffU$. 
The distinguished triangle in (\ref{ses u}) can be thus written as a fibered and cofibered diagram of derived complexes
\[
\begin{tikzcd}
\rder\pi_{\ffU *}\ffL_{\ffU} \arrow[d] \arrow[r] \arrow[rd,phantom,"\ulcorner",very near start]& \pi_{\ffU *}\ffL_{\ffU}(A)\arrow[d,"s"]\\
0 \arrow[r] & \pi_{\ffU *}\ffL_{\ffU}(A)|_{A} \text{.}\arrow[ul,phantom,"\lrcorner", very near start]
\end{tikzcd}
\]

Taking the total space $\bbV(-)=\bbR\Spec_{\ffU}\Sym^{\bullet}(-)^{\vee}$ functor gives us a homotopical fibered product 
\begin{equation}\label{eq:ffE}
\begin{tikzcd}
\bbR U\ar[r]\ar[d] \arrow[rd,phantom,"\ulcorner_h",very near start]& \UU\ar[d,"s"]\\
\ffU\ar[r,"0"] & \ffE\text{.}
\end{tikzcd}
\end{equation}
Let $E$ be the pullback of the bundle $\ffE$ by the projection $\UU\to\ffU$, and $\sigma$ be the section induced by $s$. We claim that the homotopical fibered square above implies that the square below is also homotopically fibered
\begin{equation}\label{eq:E}
\xymatrix{
\bbR U\ar[r]\ar[d]& \UU\ar[d]^{\sigma}\\
\UU\ar[r]^{0} & E\text{.}
}
\end{equation}
To see this, consider
\begin{equation}\label{eq:zeroes} 
\xymatrix{
\UU\ar[r]^{0}\ar[d]& E\ar[d]\\
\ffU\ar[r]^{0}& \ffE
} 
\end{equation}
which is obviously fibered. Stacking (\ref{eq:E}) and (\ref{eq:zeroes}) yields (\ref{eq:ffE}). Since (\ref{eq:zeroes}) and (\ref{eq:ffE}) are homotopical fibered products, (\ref{eq:E}) must also be a homotopical fibered product. 

The second part of the statement is proved in the same way, with $\wh{E}$ and $\wh{\sigma}$ coming from the following triangle over $\wh{\ffU}$:
\begin{equation}\label{ses u hat}
\rder\wh{\pi}_{\wh{\ffU} *}\wh{\ccL}_{\wh{\ffU}}\to \wh{\pi}_{\wh{\ffU} *}\wh{\ccL}_{\wh{\ffU}}(\wh{A})\xrightarrow{\wh{s}} \wh{\pi}_{\wh{\ffU} *}\wh{\ccL}_{\wh{\ffU}}(\wh{A})|_{\wh{A}}\xrightarrow{+1}\text{.}
\end{equation}
\end{proof}
The contraction $c:\Pic\to\wPic$ restricts to $c:\ffU\to\wh{\ffU}$ and induces maps $\widetilde{c}:\bbR U\to\bbR\wh{U}$ and $\widetilde{q}:\UU\to\wh{\UU}$ by the same construction as \ref{subsubsection:L,to,whL}, \ref{prop:div,rational,tails}. All these maps are in particular birational.

\begin{lemma}\label{lem:local}We have a homotopically cartesian diagram
\begin{equation}\label{local diag}
\begin{tikzcd}
\bbR U\ar[r,"i"]\ar[d,"\widetilde{c}"] \arrow[rd,phantom,"\ulcorner_h",very near start]&\UU\ar[d,"\widetilde{q}"]\\
\bbR\widehat{U}\ar[r,"\widehat{i}"]&\widehat{\UU}\text{.}
\end{tikzcd}
\end{equation}
\end{lemma}

\begin{proof}
By \cref{prop:UandUU}, it suffices to show that $E=\widetilde{q}^*\wh{E}$ and $\sigma=\widetilde{q}^*\wh{\sigma}$. 
Recall that $(E,\sigma)$ is defined by the following diagram coming from \eqref{ses u}.
\[
\begin{tikzcd}
E\ar[r]\ar[d]\ar[dr, phantom, "\ulcorner", very near start]& \Spec_{\ffU}\Sym(\pi_{\ffU *}\ccL_{\ffU}(A)|_{A})\ar[d]\\
\UU\ar[r]\ar[ur, "s"]\ar[u, "\sigma", bend left]& \ffU.
\end{tikzcd}
\]
Similarly $(\wh{E}, \wh{\sigma})$ is defined by the diagram below coming from \eqref{ses u hat}.
\[
\begin{tikzcd}
\wh{E}\ar[r]\ar[d]\ar[dr, phantom, "\ulcorner", very near start]& \Spec_{\wh{\ffU}}\Sym(\wh{\pi}_{\wh{\ffU} *}\wh{\ccL}_{\wh{\ffU}}(\wh{A})|_{\wh{A}})\ar[d]\\
\wh{\UU}\ar[r]\ar[ur, "\wh{s}"]\ar[u, "\wh{\sigma}", bend left]& \wh{\ffU}.
\end{tikzcd}
\]
For $c:\ffU\to\wh{\ffU}$ the usual contraction, we need to show that $c^*\wh{\pi}_{\wh{\ffU} *}\wh{\ccL}_{\wh{\ffU}}(\wh{A})|_{\wh{A}}\cong \pi_{\ffU *}\ccL_{\ffU}(A)|_{A}$ and that $\wh{s}\circ\widetilde{q}=s$.

To see these, start with the triangle \eqref{ses u hat} defining  $\wh{E}$ and $\wh{\sigma}$.
\[
\begin{tikzcd}
c^*\rder\wh{\pi}_{\wh{\ffU} *}\wh{\ccL}_{\wh{\ffU}}\arrow[r] \arrow[d,"\cong"]& c^*\wh{\pi}_{\wh{\ffU} *}\wh{\ccL}_{\wh{\ffU}}(\wh{A})\arrow[r,"c^*\wh{s}"] \arrow[d,"\cong"]& c^*\wh{\pi}_{\wh{\ffU} *}\wh{\ccL}_{\wh{\ffU}}(\wh{A})|_{\wh{A}}\arrow[r, "+1"]\arrow[d,"\cong"] &\  \\
 \rder{\pi}_{{\ffU} *}(k^*\wh{\ccL}_{\ffU})\arrow[r] \arrow[d, "\cong"]& {\pi}_{{\ffU} *}\left((k^*\wh{\ccL}_{\ffU})\otimes\ccO_{\ffU}({A})\right)\arrow[r]\arrow[d, "\cong"] & ({\pi}_{{\ffU} *}\left((k^* \wh{\ccL}_{\ffU})\otimes\ccO_{\ffU}({A})\right)|_{{A}} \arrow[r, "+1"] \arrow[d, "\cong"]& \ \\
 \rder{\pi}_{{\ffU} *}(\ccL_{\ffU}(\delta\ffD) )\ar[r] & {\pi}_{{\ffU} *}({\ccL_{\ffU}(\delta\ffD+A)}) \arrow[r] & \pi_{\ffU *}(\ccL_{\ffU}(A))|_{A}\ar[r, "+1"]& \ \\
 \end{tikzcd}
\]
The first set of vertical isomorphisms are by cohomology and base-change and the fact that $k^*\wh{A}=A$. The following are given by $k^*\wh{\ccL}=\kappa^*\ell^*\wh{\ccL}=\kappa^*\kappa_*\ffL(\delta \ffD)\simeq\ccL(\delta\ffD)$ (see \ref{subsub:morphism,between,secL,secLhat}).

By the requirements of our construction, $A$ does not meet $\ffD$. Then for the last term we have $\pi_{\ffU *}(\ccL_{\ffU}(A))|_{A}\cong \pi_{\ffU *}(\ccL_{\ffU}(\delta\ffD + A))|_{A}$. We conclude that $c^*$\eqref{ses u hat} is isomorphic to the following triangle:
\begin{equation}\label{c^* ses u hat}
 \rder{\pi}_{{\ffU} *}(\ccL_{\ffU}(\delta\ffD) )\to {\pi}_{{\ffU} *}({\ccL_{\ffU}(\delta\ffD+A)})\xrightarrow{c^*\wh{s}} \pi_{\ffU *}(\ccL_{\ffU}(\delta\ffD+A))|_{A}\xrightarrow{+1}\text{.}
\end{equation}
Now we compare $c^*$\eqref{ses u hat} =\eqref{c^* ses u hat} to  \eqref{ses u}. Twisting by $\delta\ffD$ induces a map between them
\[
\begin{tikzcd}
 \rder\pi_{\ffU *}\ccL_{\ffU}\arrow[r] \arrow[d] & \pi_{\ffU *}\ccL_{\ffU}(A)\arrow[r, "s"] \arrow[d, "f"] & \pi_{\ffU *}\ccL_{\ffU}(A)|_{A}\arrow[r, "+1"] \arrow[d, "\cong"] & \ \\
 \rder{\pi}_{{\ffU} *}(\ccL_{\ffU}(\delta\ffD) )\ar[r]\ar[d] & {\pi}_{{\ffU} *}({\ccL_{\ffU}(\delta\ffD+A)}) \arrow[r, "c^*\wh{s}"]\ar[d] & \pi_{\ffU *}(\ccL_{\ffU}(\delta\ffD+A))|_{A}\ar[r, "+1"]\ar[d] & \ \\
  \rder{\pi}_{{\ffU} *}(\ccL_{\ffU}(\delta\ffD) )|_{\delta\ffD}\ar[r, "\cong"] & {\pi}_{{\ffU}*}({\ccL_{\ffU}(\delta\ffD+A)})|_{\delta\ffD} \arrow[r] & 0\text{.} & \ \\
\end{tikzcd}
\]
The vertical map $f$ above is as follows:
\[
\widetilde{q}: \UU\xrightarrow{f} c^*\wh{\UU}\to \wh{\UU}\text{.}
\]
This shows that $c^*\wh{\pi}_{\wh{\ffU} *}\wh{\ccL}_{\wh{\ffU}}(\wh{A})|_{\wh{A}}\cong \pi_{\ffU *}\ccL_{\ffU}(A)|_{A}$ and that $\wh{s}\circ\widetilde{q}=s$, completing the proof.

\end{proof}
Now we have opens $W\subset\UU'\subset\UU$ and $\wh{W}\subset\wh{\UU'}\subset\wh{\UU}$. From the definitions of $W$ and $\wh{W}$ we see that $\widetilde{q}$ restricts to a map $q:W\to\wh{W}$.
\begin{lemma}\label{lemma:WandUU}
We have a commutative diagram
\[
\begin{tikzcd}
W\ar[r]\ar[d,"q"]& \UU\ar[d,"\widetilde{q}"]\\
\wh{W}\ar[r] & \wh{\UU}\text{.}
\end{tikzcd}
\]

\end{lemma}
\begin{proof}
We only need to check that the image of $W$ is contained in $\wh{W}$. This follows by comparing the stability conditions in \cref{construction W} and \cref{construction W hat}.
\end{proof}

\begin{lemma}\label{lemma:comp of stability}
We have that $\bbR V$ is an \'{e}tale neighbourhood of $\xi$ in $\RGw$.
\end{lemma}
\begin{proof}
Recall that $\bbR V$ is defined by
\[
\begin{tikzcd}
\bbR V\ar[r]\ar[d]\ar[dr,phantom,"\ulcorner^h",very near start]& W\ar[d]\\
\bbR U\ar[r, "m_A"]&\UU\text{.}
\end{tikzcd}
\]
The first observation is that $W$ is open in $\UU$. Indeed, we have defined an open analytic subset (\'{e}tale neighborhood) $\UU'\subset \UU$ around $\zeta$, the image of $\xi$ under the morphism $m_A$ induced by tensoring with $\ccO(A)$. In this, $W$ is cut out in \cref{construction W} by imposing open stability conditions. Since the point $\zeta$ was the image of a stable point $\xi$, the stability conditions hold for it. So $W\subset \UU$ is an \'{e}tale neighborhood of $\zeta$.

Thus, $\bbR V\to\bbR U$ is also an open embedding, and $\xi\in\bbR V$. On the other hand, we have an open subset $\mathfrak{V}$ of $\bbR U=\bbR\Sec_{\ffU}(\ffL^{\oplus r+1}/\ffC_{\ffU})$ given by
\[
\mathfrak{V}=\RGw\times_{\bbR\Sec_{\Pic}(\ffL^{\oplus r+1}/\ffC)}\left(\bbR U\times_{\UU}\UU'\right)\text{.}
\]
We want to show that $\bbR V$ and $\mathfrak{V}$ are equivalent. First, we show that their truncations are isomorphic, that is $t_0(\mathfrak{V})=V'\simeq t_0(\bbR V)=V$. It suffices to show that the conditions of \cref{construction W} are equivalent to the stability conditions of stable maps on points in the image of (the truncation of) $m_A$. We recall them here for the reader's convenience.
At any point $(C,L,s_0,\dots,s_r)$ of $U\times_{\UU}\UU'$ \cref{conditions_st_maps_as_sections} state it is in $V'$ iff the following hold:
\begin{enumerate}
\item the bundle $\omega_C^{\log}\otimes L^{\otimes 3}$ is ample and
\item the linear system $(L,s_0,\dots,s_r)$ has no base points.
\end{enumerate}
On the other hand, at any point $(C,L(A),w_0,\dots ,w_r)$ of $\UU'$, \cref{construction W} states it is in $W$ iff:
\begin{enumerate}[label=(\roman*)]
\item the base locus $BL(w)$ of $w_0, \dots , w_r$ is discrete and disjoint from the special points of $C$, 

\item the subset $BL(w)_{L}$ of the base locus is empty and
\item the line bundle $\omega_{C}^{\log}\otimes L^{\otimes 3}$ is ample.
\end{enumerate}
Conditions (1) and (iii) are clearly equivalent. We want to show that for points in $m_A(U)$ condition (ii) implies condition \eqref{condition:no_base_points_st_maps_as_sections}, and that \eqref{condition:no_base_points_st_maps_as_sections} holding for points of $U$ implies (i) and (ii), that is: their image under $m_A$ lies in $W$.

Let $(C,L,s_0,\dots,s_r)$ be a point in $U$, 
\[
m_A(C,L,s_0,\dots,s_r)=(C,L(A),w_0,\dots ,w_r)
\]
where $w_i$ is the image of $s_i$ under $H^0(C,L)\to H^0(C,L(A))$. 
Condition (ii) implies that the $BL(w)\subset A$. On the other hand, $\wh{A}$ and $\wh{\ffU}$ were chosen so that $BL(\wh{s})$ does not intersect $\wh{A}$ on the open $\wh{U}$, then also $BL(s)$ does not intersect $A$ in $U$. Then we see that $BL(s)$ must be empty.

Conversely, if $BL(s)$ is empty, $BL(w)$ must be contained in $A$, which implies both condition (i) and (ii).

We have that $V\simeq V'=\Gw\times_{\Sec_{\Pic}(\ffL^{\oplus r+1}/\ffC)}(U\times_{\UU}\UU')$.

Finally, we observe that the derived structures on both \'{e}tale neighborhoods of $\xi$ in $\bbR\Sec_{\Pic}(\ffL^{\oplus r+1}/\ffC)$ are compatible, since \'{e}tale maps are flat. Hence, we have that $\bbR V$ is equivalent to $\ffV$, as desired.

\end{proof}

\begin{lemma}\label{lemma:comp of stability hat}
We have that $\bbR \wh V$ is a neighbourhood of $\wh{\xi}$ in $\Qm$.
\end{lemma}
\begin{proof}
This is similar to the proof of \ref{lemma:comp of stability}.

\end{proof}

\begin{lemma}\label{lemma: local V, W}
We have a homotopical cartesian diagram
\[
\begin{tikzcd}
\bbR V\arrow[r]\ar[d, "c" '] \arrow[dr,phantom,"\ulcorner^h", very near start]& W\arrow[d, "q"]\\
\bbR\wh{V}\arrow[r] & \wh{W}\text{.}
\end{tikzcd}
\]
\end{lemma}
\begin{proof}
We denote the restriction of $(E,\sigma)$ to $W$ by $(F,\theta)$. By \cref{prop:UandUU} we had $\bbR U=\bbR Z(\sigma)$ i.e., $\ccO_{\bbR U}=\Kos(E,\sigma)$. Since $\bbR V=\bbR U\times_{\UU}W$, we have  
\begin{equation}\label{V is kos}
    \bbR V = \bbR Z(\theta)\text{.}
\end{equation}

By \cref{lemma:WandUU} and \cref{lem:local} the restriction of $(\wh{E},\wh{\sigma})$ from $\wh{\UU}$ to $\wh{W}$ is $(q^*F,q^*\theta)$. From \cref{prop:UandUU} we had that $\bbR\wh{U}=\bbR Z(\wh{\sigma})$. Then by the definition of $\bbR \wh{V}$ we have
\begin{equation}\label{V hat is kos}
    \bbR \wh{V} = \bbR Z(q^*\theta)\text{.}
\end{equation}
\end{proof}

%%%%%%%%%%%%%%%%%%%%%%%%%%%%%%%%%%%%%%%%%%%%%%%%%%%%%%%%%%%%%%
\section{Main theorem}\label{sec:main,thm,pushforward}
%%%%%%%%%%%%%%%%%%%%%%%%%%%%%%%%%%%%%%%%%%%%%%%%%%%%%%%%%%%%%%%%
We are now ready to prove our main theorem on the derived push-forward of the structure sheaf of $\RGw$.
Just recall that contracting rational tails gives a morphism
\[
\oc: \RGw \to \RQm
\]
 To prove our main theorem (See \cref{thm:pushforward}) that is
  \[
 \oc_* \ccO_{\RGw} = \ccO_{\RQm}.
 \]
it is enough to do it locally. That's why \S \ref{subsec:local,embeddings} is useful as we have a local picture for $\oc$.
In \cref{subsec:properness}, we will prove that $q:W \to \wh W$ is proper and birational (See Proposition \ref{prop:properness,q,birational}. 
In \ref{subsec:derived,push-forward,last-sub},  we use the Zariski Main theorem to prove that 
\begin{align}\label{eq:push,struc=struc}
    q_{*}  \ccO_{W}=\ccO_{\wh W}
\end{align}
(see proof of Lemma \ref{lemma:local,statement}).
Then by cohomology and base change, we prove our main theorem (see \cref{thm:pushforward}).

%%%%%%%%%%%%%%%%%%%%%%%%%%%%%%%%%%%%%%%%%%%%%%%%%%%%%%%
\subsection{Properness of $q$}\label{subsec:properness}
%%%%%%%%%%%%%%%%%%%%%%%%%%%%%%%%%%%%%%%%%%%%%%%%%%%

\begin{proposition}\label{prop:properness,q,birational}
The morphism $q: W\to \wh W$ is proper and birational.
\end{proposition}
\begin{proof}

Birationality follows from the fact that $\widetilde{q}:\UU\to\wh{\UU}$ is birational and $W$, $\wh{W}$ are open subsets of $\UU$ and $\wh\UU$ respectively.

We use the valuative criterion to prove properness. Consider the following diagram
\begin{equation}\label{unique-extension}
\begin{tikzcd}
\bbA^1\setminus \{0\}\ar[d] \arrow[r,"\varphi^\circ"]&W \ar[d,"q"] \\
\bbA^1 \arrow[r,"\widehat{\varphi}"] \arrow[ur,dashed,"\exists !"]&\wh{W}\text{.}
\end{tikzcd}
\end{equation}
The morphism $\wh{\varphi}$ above is given by a family $\wh{\ccC}\to \bbA^1 $, together with a line bundle $\wh\ccG$ and sections $(\wh w_0,\ldots, \wh w_r)$. We denote by $\wh{\ccC}^\circ$ the restriction of $\wh{\ccC}$ to $\bbA^1\backslash\{0\}$. The morphism $\varphi^\circ$ gives a family $(\ccC^\circ,\ccG^\circ, w_0^\circ,\ldots, w_r^\circ)$ such that \[q(\ccC^\circ, \ccG^\circ, w_0^\circ,\ldots, w_r^\circ)=(\wh{\ccC}^\circ, \wh\ccG^\circ, \wh w_0^\circ,\ldots, \wh w_r^\circ)\text{.}
\]
Here, by abuse of notation we denoted by $q$ the map induced by $q: W\to \wh W$. 

In the following we show that there exists a unique morphism $\varphi$ which extends $\varphi^\circ$ and makes diagram (\ref{unique-extension}) commute. In concrete terms, this amounts to finding $(\ccC, \ccG, w_0,\ldots, w_r)$ a family over $\bbA^1$ which extends $(\ccC^\circ,\ccG^\circ, w_0^\circ,\ldots, w_r^\circ)$ and such that 
\[q(\ccC, \ccG, w_0,\ldots, w_r)=(\wh{\ccC}, \wh\ccG, \wh w_0,\ldots, \wh w_r)\text{.}
\]

{\bf Existence.}

By definition, $\wh W$ parameterises  tuples $(\wh C, \wh L(\wh A), \wh w_0,\ldots \wh w_r)$, subject to the non-degeneracy condition in \ref{construction W hat}. This shows that $\wh W$ is a subset of $\overline{\ccQ}_{g,n} (\bbP^r,d+a)$. Let $\widetilde{W}$ be the fibre product
\[\xymatrix{\widetilde{W}\ar[r]\ar[d]& \overline{\ccM}_{g,n} (\bbP^r,d+a)\ar[d]^{\oc}\\
\wh W\ar[r]&\overline{\ccQ}_{g,n} (\bbP^r,d+a)\text{.}
}
\]

In the following we construct a morphism $\widetilde W\to W$ such that $\widetilde W\to\wh W$ factors through $\widetilde W\to W$. The construction is the one in Theorem \ref{prop:div,rational,tails} with minor modifications. Let $(\widetilde{\ffC},\widetilde{\ffG})$ denote the universal curve and universal bundle on $\Pic_{d+a}$.
We have 
\begin{align*}
&\widetilde W\subset\Sec_{\Pic_{d+a}}(\widetilde{\ffG}^{\oplus r+1}/\widetilde{\ffC})\\
& W\subset\Sec_{\Pic_d}(\ffL(A)^{\oplus r+1}/\ffC)= W\subset\Sec_{\Pic_d}(\ffG^{\oplus r+1}/\ffC)\\
& \wh{W}\subset\Sec_{\wh{\Pic}_d}(\wh{\ffL}(\wh{A})^{\oplus r+1}/\wh{\ffC})= \Sec_{\wh{\Pic}_{d+a}}({\wh{\ffG}_{d+a}}^{\oplus r+1}/{\wh{\ffC}_{d+a}})\simeq \Sec_{\wh{\Pic}_{d+a}}({\wh{\ffL}_{d+a}}^{\oplus r+1}/{\wh{\ffC}_{d+a}})
\end{align*}
where $\wh{\ffC}_{d+a},\wh{\ffL}_{d+a}$ are the universal curve and line bundle over $\wPic_{d+a}\cong\wPic_d$. The isomorphism $\wPic_d\to\wPic_{d+a}$ is given by $(\wh{C},\wh{L})\mapsto (\wh{C},\wh{L}\otimes\ccO_{\wh{C}}(\wh{A}))$.

\textit{Claim.} Let $(\widetilde{C},\widetilde{G},\widetilde{w})\in \widetilde{W}$ and let $RT$ be a rational tail of $\widetilde{C}$. We have 
\[ RT=RT_{\tilde{\ccL}}\sqcup RT_{\tilde A} \text{.}
\]

The claim follows from the fact that $\un{\widetilde{w}}= \un{\wh{w}}$ outside the exceptional locus of $\widetilde\ccC\to \wh\ccC$ and the fact that $BL(\un{w})=BL(\un{w})_L\sqcup BL(\un{w})_A$.  We need to show that the labelling on $\wh{W}$ lifts to a labelling of the rational tails of the universal curve of $\widetilde{W}$. We have that $p:\widetilde{\ffC}\to \ffC$ contracts rational tails.  Since the base loci of $w_L$ and $w_A$ are disjoint we get that the base loci of $p^{-1}w_L$ and $p^{-1}w_\ccA$ are disjoint. Moreover, since the base loci of $w_L$ and $w_A$ form disconnected components, the same holds about their inverse images. This proves the claim.

By possibly shrinking $\wh{W}$ and changing the basis of $\bbP^r$, we have a divisor $Z(\un{\tilde{w}})_{\widetilde A}$ on $\tilde\ccC$, which we denote by $\widetilde{A}$. Let $\tilde\ccL$ denote $\tilde\ccG\otimes \ccO(-\widetilde{A})$.

Let $S$ be a scheme. In the following we contract rational tails of $\widetilde\ccC_S$ which intersect $\widetilde A$. Let $D_A$ be the divisor of $\widetilde\ccC_S$, which consists of rational tails $RT_{\widetilde{A}}$ and let $\delta_A$ be the degree of $\widetilde\ccL_S$ restricted to the tails. We have that $\ccL_S':=\widetilde\ccL_S(\delta_A D_A)$ is trivial along the exceptional locus $D_A$ and base point free. Let us define
\[
\ccC=\Proj \sum _nH^0(\widetilde{\ccC}_S, (\ccL_S')^{\otimes n})\text{.}
\]
 Let $\kappa: \widetilde{\ccC}_S\to\ccC_S$ and let $\ccL_S:=\kappa_*\ccL'_S$. Since $\ccL'_S$ is trivial along $D_A$, Lemma 7.1 in \cite{Popo-roth-stable-maps-Quot-2003}  implies that $\ccL_S$ is a line bundle. In the same way as we did in the proof of Theorem \ref{prop:div,rational,tails}, we construct $(w_0,\ldots,w_r)$ sections of $\ccL$. We thus obtain a surjective morphism $\widetilde W\to W$. It can be seen that $\widetilde{W}\to \wh W$ factors through $\widetilde{W}\to W$.

 Since $\widetilde{W}\to W$ is surjective, there exists a (non unique) family of maps \[(\widetilde{\ccC}^\circ,\widetilde\ccG^\circ, \widetilde w_0^\circ,\ldots, \widetilde w_r^\circ)\]
such that 
\[q(\widetilde{\ccC}^\circ,\widetilde\ccG^\circ, \widetilde w_0^\circ,\ldots, \tilde w_r^\circ)=(\ccC^\circ,\ccG^\circ, w_0^\circ,\ldots, w_r^\circ)\text{.}
\]
Equivalently, we have a family $\widetilde{\varphi}^{\circ}:\bbA^1\backslash\{0\}\to \widetilde{W}$ which commutes with $\varphi^{\circ}$. Since $\Gw$ is proper, we have that $\oc$ is proper. This implies that $\widetilde{W}\to \wh W$ is proper. This shows that $\widetilde{\varphi}^{\circ}$ extends (uniquely) to $\widetilde{\varphi}: \bbA^1\to \widetilde{W}$, and thus we obtain a family of curves $\widetilde{\ccC}\to\bbA^1$, a line bundle $\widetilde\ccG$ on $\tilde{\ccC}$ and  $(\widetilde{w}_0,\ldots \widetilde{w}_r)$ global sections of $\ccG$ which makes the following diagram commute. 

\begin{equation}
\begin{tikzcd}
\bbA^1\ar[dd] \arrow[r,"\widetilde{\varphi}"]&\widetilde W \ar[d] \\
&W\ar[d]\\
\bbA^1 \arrow[r,"\widehat{\varphi}"]&\wh{W}
\end{tikzcd}
\end{equation}

To complete the following diagram
\begin{equation}
\begin{tikzcd}
\bbA^1\ar[dd] \arrow[r,"\varphi"]&\widetilde W \ar[d] \\
&W\ar[d]\\
\bbA^1 \arrow[r,"\widehat{\varphi}"]\arrow[ur,dashed,"\exists !"]&\wh{W}
\end{tikzcd}
\end{equation}
we consider the image of $(\widetilde{\ccC}, \widetilde{\ccL}, \widetilde{w}_0,\ldots \widetilde{w}_r)$ in $W$.
\ \\

{\bf Uniqueness.}
In notations as before, we have morphisms $\widetilde{W}\to W \to \wh{W}$.
We have that $\widetilde{W}\to \wh{W}$ is separated, because by construction it is proper. The map $\widetilde{W}\to W$ is surjective and proper by the discussion above. With this, we are under the assumptions of \cite[{Tag 09MQ}]{stacks-project}. This shows that $W\to \wh{W}$ is separated.

\end{proof}

\begin{lemma}\label{projective}
Let $|W|$ and $|\wh W|$ denote the coarse moduli spaces of $W$, respectively $\wh W$. The morphism $|q|:|W|\to |\wh{W}|$ is projective.
\end{lemma}
\begin{proof}
Since $\wh{W}$ is separated, it suffices to show that $|q|:|W|\to |\wh{W}|$ is projective. Recall that $W$ is open in a DM stack $\pi_*\ffL(A)$ defined by the following stability conditions, for a point $(C,L(A),w_0,\ldots w_r)$ 
\begin{enumerate}
\item the bundle $\omega^{\log}_C\otimes L^{\otimes 3}$ is ample
\item The base locus $BL(w)=\bigcap_{i=0}^{r}Z(w_i)$ has dimension 0 and is distinct from marked points and nodes.
\end{enumerate}
To show projectivity, we can follow the proof of \cite{cornalba_projective}. We sketch here the necessary modifications, trying to adhere to the notation of the original proof as much as possible.

A family $F:S\to W$ consists of a pre-stable curve $\pi_S:C_S\to S$ with $n$ marked points $(x_1,\ldots x_n):S\to C_S$, a distinguished divisor $A_S$ of degree $a$, a line bundle $L_S$ of degree $d$ and sections $(w_0,\ldots w_r)$ of $L_S(A_S)$.
We can define define a line bundle on $S$ by
\[
\ccV_F=\langle\omega_{\ffC_{S}}^{\log}\otimes\ccL_{S}(A_S)^{\otimes 3},\omega_{\ffC_{S}}^{\log}\otimes\ccL_{S}(A_S)^{\otimes 3}\rangle
\]
using Deligne's bilinear pairing, explicitly for $V_F=\omega_{\ffC_{S}}^{\log}\otimes\ccL_{S}(A_S)^{\otimes 3}$, we have
\[
\ccV_F=\det \rder\pi_{S *}\ccO_S\otimes (\det\rder\pi_{S *} V_F)^{\otimes -2}\otimes \det \rder\pi_{S *} (V_F\otimes V_F).
\]
We want to show this bundle $\ccV_F$ is ample. Following Cornalba's approach, which relies on Seshadri's criterion, it suffices to show that there exists a constant $\alpha=\alpha(g,n,r,d)>0$ such that for any non-isotrivial family $F$ over an integral complete curve $S$, since we have already proved that $|q|:|W|\to |\wh{W}|$ is proper.
\begin{equation}\label{eq:toproveampleness}
(V_F\cdot V_F)\geq \alpha m(S)
\end{equation}
where $m(S)$ denotes the maximum multiplicity of points in $S$. 

Since the number of nodes of the curve $C_S$ is bounded in terms of $(g,n,d,r)$ for any family, we may reduce to the case of a family $F$ whose generic curve is smooth, as in the original proof. Now the idea is to add marked points to $C_S$ to obtain a stable domain curve. Since we do not have a well-defined map to $\bbP^r$, we can use the sections to add $3(d+a)$ marked points. Indeed, by taking linear combinations of the sections $(w_0\ldots w_r)$ we may assume that we have a linearly independent set $(w_0,w_1,w_2)$ such that for $i\in\{0,1,2\}$ the following conditions hold (c.f. \cite[Lemma 2]{cornalba_projective} note that our condition (iii) is equivalent to (ii), (iii) and (v)):
\begin{enumerate}[label=(\roman*)]
\item $Z(w_i)$ does not contain components of the fiber of $\pi_S$
\item $Z(w_i)$ does not contain $x_j$ for $j=1,\ldots n$
\item $Z(w_i)$ consists of $d+a$ distinct, non-special points on all the fibers of $\pi_S$ which are singular or lie over singular points of $S$. 
\end{enumerate}
We take 
\begin{align*}
&Z(w_0)=x_{n+1}+\dots x_{n+d+a}\\
&Z(w_1)=x_{n+d+a+1}\dots x_{n+2(d+a)}\\
& Z(w_2)=x_{n+2(d+a)+1}\dots x_{n+3(d+a)}
\end{align*}
where we may assume, up to some finite base change of bounded degree, that $(x_{n+1},\dots , x_{n+3(d+a)})$ are distinct as sections of $\pi_S$ and distinct from the original sections $(x_1,\ldots x_n)$. Now, on smooth fibers of $\pi_S$, some of the $x_i$'s may still meet, indeed they will if the $w$'s defined a linear system with non-empty base-locus. We may proceed to resolve them as in \cite[Proof of Lemma 2]{cornalba_projective} and obtain a family of stable curves
\[
F'=\left\{C'_S\to S, x'_1,\dots , x'_{n+3(d+a)}\right\}
\]
with
\begin{equation}\label{eq:inequalityVF}
(V_F\cdot V_F)=(\omega_{C_S}(D)\cdot\omega_{C_S}(D))\geq (\omega_{C'_S}(D')\cdot \omega_{C'_S}(D')).
\end{equation}
where $D=\sum_{i=1}^{n+3(d+a)}x_i$ and $D'=\sum_{i=1}^{n+3(d+a)}x'_i$. We may assume $F'$ is a non-isotrivial family, otherwise we proceed as in \cite[Lemma 3]{cornalba_projective}. Now, $(S,C'_{S},D')$ is a non-isotrivial stable family, so $\kappa_1=\pi_{S' *}(\omega_{C_S'}(D')^{\otimes 2})$ is ample on $S$, thus by Seshadri's criterion and \eqref{eq:inequalityVF} we have the required $\alpha$ to conclude that \eqref{eq:toproveampleness} holds.

\end{proof}

%%%%%%%%%%%%%%%%%%%%%%%%%%%%%%%%%%%%%%%%%%%%%%%%%%%%%
\subsection{Derived push-forward}\label{subsec:derived,push-forward,last-sub}
%%%%%%%%%%%%%%%%%%%%%%%%%%%%%%%%%%%%%%%%%%%%%%%%%%%%%
In this subsection, we will prove the main theorem of this paper that is:

\begin{theorem}\label{thm:pushforward}
For any, $g,n$ and $d$, we have that 
\[
\oc_* \ccO_{\RGw} = \ccO_{\RQm} \mbox{ in } \ccD^b(\Coh(\bbR\overline{\ccQ} (\bbP^r,d))\text{.}
\]
\end{theorem}

\begin{remark}

At the level of virtual classes, we have that
\[
\oc_* [\Gw]^\vir = [\Qm]^\vir \text{.}
\]
This was proven in \cite{Fontanine-QMap2010}, \cite{Marian-Oprea-Pand-moduli-stable-2011} and \cite{Cristina-stable-maps-quotients-2014}.
\end{remark}

We deduce the following corollary.
\begin{corollary}\label{coro:K,GW,QM}
The $G$-theoretic Gromov-Witten invariants and the $G$-theoretic quasimaps invariants are equal.
\end{corollary}

\begin{remark}
Recall that by definition
\[
K(X):= K(\Perf(X)) \mbox{ and }  G(X):=K( \ccD^b(\Coh(X))\text{,}
\]
which are denoted respectively $K^{\circ}(X)$ and $K_{\circ}(X)$ by Lee in~\cite{Lee-QK-2004}.
When $X$ is smooth, $\ccD^b(\Coh(X))$ and $\Perf(X)$ coincide.
\end{remark}

From the Lemma \ref{lemma: local V, W}, we want study the morphism $q:W\to \widehat{W}$. Notice that there are both smooth stacks.

 \begin{proposition}
 \label{prop:local,statement}
We have
\begin{align}
\rder^0 q_*\mathcal{O}_W&=\mathcal{O}_{\widehat{W}} \mbox{ in } \mathcal{D}^b(Coh(\widehat{W}) \label{eq: RO}\\
\rder^i q_*\mathcal{O}_W&=0 \mbox{ for } i>0\text{.} \label{eq:Ri}
\end{align}
\end{proposition}

Before proving this proposition, we need some lemmas.

Denote by $|q|:|W|\to |\widehat{W}|$ the coarse moduli map of $q:W\to \widehat{W}$. We have the following commutative diagram

\[
\begin{tikzcd}
W \arrow[rrd, bend left,"q"] \arrow[rdd,bend right] \arrow[rd,"\alpha"] && \\
&\vert W \vert \times_{\vert \widehat{W} \vert}\widehat{W} \arrow[r,"q'"] \arrow[d] \arrow[rd,phantom,"\ulcorner", very near start] & \widehat{W} \arrow[d]\\
& \vert W \vert \arrow[r,"|q|"]& \vert \widehat{W}\vert\text{.}
\end{tikzcd}
\]

Notice that $|q|$ is representable so is $q'$.

\begin{proof}[Proof of Proposition \ref{prop:local,statement}]
As $q=q'\circ \alpha$, and from Lemmas \ref{lemma:vistoli} and \ref{lemma:local,statement}, we deduce the proposition.
\end{proof}

\begin{lemma}\label{lemma:vistoli}
 \begin{enumerate}
    \item The morphism $\alpha$ is finite.
    \item We have $\rder^0\alpha_*\mathcal{O}_W=\mathcal{O}_{|W|\times_{|\widehat{W}|}\widehat{W}}$.
    \item For $i>0$, we have $\rder^i\alpha_*\mathcal{O}_W=0$.
\end{enumerate}
\end{lemma}

\begin{proof}[Proof of Lemme \ref{lemma:vistoli}]
The third statement follows from the first.

The first and second statement follows from Theorem 3.1 in \cite{AOV-twisted-stable} if we have that $q'$ is the relative coarse moduli space of $q:W\to \widehat{W}$ in the sense of \cite[Definition 3.2]{AOV-twisted-stable}. Let's prove that $q'$ is the relative coarse moduli space of $q$.

From the proof of \cite[Theorem 3.1]{AOV-twisted-stable}, we get that the relative coarse moduli space of $q:W\to \widehat{W}$ is given by the following construction.
Let $[R\rightrightarrows U]$ by a smooth presentation of $\widehat{W}$. Then the relative coarse moduli space of $q:W\to \widehat{W}$ has the following presentation (see proof of  Theorem 3.1 in \cite{AOV-twisted-stable})
\[
\Bigl[\left| W \times_{\widehat{W}}R\right| \rightrightarrows \left|W\times_{\widehat{W}}U\right|\Bigr]
=|W|\times_{|\widehat{W}|}[R\rightrightarrows U]
=|W|\times_{|\widehat{W}|}\widehat{W}\text{.}
\]
\end{proof}

\begin{lemma}\label{lemma:Zariski,main,thm}
 \begin{enumerate}
    \item The morphism $q'$ is representable, birational and projective.
    \item We have $\rder^0q'_*\mathcal{O}_{|W|\times_{|\widehat{W}|}\widehat{W}}=\mathcal{O}_{\widehat{W}}$.
    \item For $i>0$, we have $\rder^iq'_*\mathcal{O}_{|W|\times_{|\widehat{W}|}\widehat{W}} =0$.
\end{enumerate}
\end{lemma}

\begin{proof}[Proof of Lemma \ref{lemma:Zariski,main,thm}]
The morphism $q'$ is representable  birational an proper by Lemma \ref{prop:properness,q,birational}.
Hence by Zariski Main theorem, we have that 
\[
\rder^0q'_* \ccO_{|W|\times_{|\widehat{W}|}\widehat{W}}=\ccO_{\widehat{W}}\text{.}
\]

The fact that $q'$ is moreover projective implies that (See  \cite[Theorem 1.1]{higher-pushforward-vanishing-2015-Rulling} or  Hironaka \cite[p.144]{Hironaka-64}).
For $i>0$, we have \[\rder^i q'_*\mathcal{O}_{|W|\times_{|\widehat{W}|}\widehat{W}} =0\text{.}\]
\end{proof}  

From the Lemma \ref{lemma: local V, W}, we have the local picture statement that is the diagram is homotopically cartesian:

\begin{equation}\label{local,diag,bis}
 \begin{tikzcd}
\bbR V\arrow[r,"i"]\arrow[d,"\oc"] \arrow[dr, phantom, "\ulcorner ", very near start] & W\arrow[d,"q"] \\ \bbR\widehat{V}\arrow[r,"\widehat{i}"]&\widehat{W}\text{.}
\end{tikzcd}
\end{equation}

\begin{lemma}\label{lemma:local,statement}
We have $\rder\oc_*\mathcal{O}_{\mathbb{R}V}=\mathcal{O}_{\mathbb{R}\widehat{V}}$.
\end{lemma}
\begin{proof}
 This follows from derived base change:
\begin{align*}
\rder\oc_*\mathcal{O}_{\mathbb{R}V}&=\rder\oc_* \lder i^*\mathcal{O}_{W}\\
&= \lder\widehat{i}^* \rder q_*\mathcal{O}_{W}\\
&= \lder\widehat{i}^* \mathcal{O}_{\widehat{W}} \mbox { by \eqref{eq: RO} and \eqref{eq:Ri} }\\
&=\mathcal{O}_{\mathbb{R}\widehat{V}}\text{.}
\end{align*}

\end{proof}

 \begin{proof}[Proof of Theorem \ref{thm:pushforward}]
The morphism $\oc:\RGw \to \RQm.$ Gives a morphism of structure sheaves 
\[
c: \ccO_{\RGw} \to  \oc_*\ccO_{\RQm}\text{.}
\]
To prove that it is an isomorphism, it is enough to prove it étale locally. That's exactly what we have done in \S \ref{subsec:local,embeddings}. Hence we are in the situation of diagram \eqref{local,diag,bis}, the Lemma \ref{lemma:local,statement} finishes the proof.
 \end{proof}

\bibliography{mybib}
\bibliographystyle{alpha}
\end{document}